\pdfoutput=1
\documentclass[11pt]{amsart}
\usepackage{geometry}                
\usepackage[parfill]{parskip}   
\usepackage{graphicx,xcolor,amsthm}
\usepackage{amssymb}
\usepackage{epstopdf}
\usepackage{pdfsync}
\usepackage{caption}
\usepackage{subcaption}
\usepackage{mathtools}
\usepackage{enumitem}
\usepackage{fullpage}
\usepackage{todonotes}
\usepackage{pinlabel}
\theoremstyle{plain} 
\newtheorem{thm}{Theorem}

\newtheorem{cor}[thm]{Corollary} 
\newtheorem{lem}[thm]{Lemma} 
\newtheorem{prop}[thm]{Proposition}
 
\newtheorem{rem}[thm]{Remark}

\theoremstyle{definition} 
\newtheorem{defn}{Definition}
\newtheorem{ex}{Example}

\title{Pseudo-Legendrian and Legendrian Simplicity of Links in 3-Manifolds}
\author{Patricia Cahn, Rima Chatterjee, Vladimir Chernov}

\begin{document}
\maketitle

\begin{abstract}
 We construct infinite families of non-simple isotopy classes of links in overtwisted contact structures on $S^1$-bundles over surfaces.
 These examples include: (1) a pair of Legendrian links that are not Legendrian isotopic, but which are isotopic as framed links, homotopic as Legendrian immersed multi-curves, and have Legendrian-isotopic components and (2) a pair of Legendrian links that are not Legendrian isotopic, but are isotopic as framed links, homotopic as Legendrian immersed multi-curves, and which are link-homotopic as Legendrian links. Moreover, we construct examples showing that both of these non-simplicity phenomena can occur in the same smooth isotopy class. To construct these examples, we develop the theory of links transverse to a nowhere-zero vector field in a 3-manifold, and construct analogous examples in the category of links transverse to a vector field.
\end{abstract}

\section{Introduction}

There has been much recent work on classification problems for Legendrian links, with many exciting new questions laid out in the work of Dalton-Etnyre-Traynor \cite{dalton2021legendrian}.  The second author and Etnyre--Min--Rodewald showed the existence of stabilized Legendrian links which are topologically isotopic and  component-wise Legendrian isotopic, but not Legendrian isotopic as links \cite{chatterjee2025cable}. The first example with non-stabilized components was observed in \cite{jordan2006generating}, where the authors used generating functions to distinguish such links. Note that all of these results are in tight contact manifolds.  This paper gives similar phenomena in general overtwisted manifolds for links whose components are loose. This is surprising as overtwisted manifolds and loose knots and links are known to be more flexible than tight manifolds; on the other hand, this flexibility allows us to construct examples exhibiting non-simplicity phenomena that have not yet been studied in tight manifolds. The foundational works of Eliashberg~\cite{eliashberg1989classification,eliashberg1992contact} introduced tight and overtwisted contact manifolds and first studied the flexibility phenomenon for overtwisted contact structures.

In $\mathbb{R}^3$ with the standard contact structure, an isotopy class of knots or links is {\it Legendrian simple} if any two Legendrian representatives of that class with the same (component-wise) Thurston-Bennequin and rotation numbers are Legendrian isotopic. The question of which isotopy classes are Legendrian simple is more complex when one considers arbitrary contact 3-manifolds $(M,\xi)$ with a co-oriented contact structure.  First, the Thurston-Bennequin numbers of the components are only well-defined when the components are null-homologous, and the rotation numbers are only well-defined when the components are null-homologous or when the contact structure is parallelizable.  One can instead formulate the definition of a simple isotopy class using two {\it generalized classical invariants}: the framed isotopy class replaces the Thurston-Bennequin number, and the homotopy class in the category of Legendrian immersions replaces the rotation number. We then define an isotopy class of framed knots or framed links to be {\it Legendrian simple} if any two Legendrian representatives of that class that are homotopic through Legendrian multi-curves are also isotopic as Legendran links. Analogously, we say a framed isotopy class of knots/links is {\it non-simple} if there are Legendrian representatives of the class that are homotopic as Legendrian multi-curves but are not Legendrian isotopic as knots/links.  When considering links, rather than knots, the notion of simplicity has various refinements, and exploring these notions of simplicity is the main goal of this paper.

The study of Legendrian-simple and non-simple isotopy classes of knots (rather than links), in terms of generalized classical invariants, was carried out by the first and third author in \cite{CC}.  In particular, the authors constructed infinite families of Legendrian-simple and non-simple isotopy classes of knots in overtwisted contact structures $S^1$-bundles over surfaces, and gave sufficient conditions for an isotopy class to be Legendrian simple in an arbitrary co-oriented contact manifold $(M,\xi).$
 The examples in \cite{CC} are constructed using homotopy-theoretic data from the space of loops in $M$, and the homotopy class of $\xi$; this homotopy-theoretic answer is more subtle than one might expect.  The goal of this article is to establish similar machinery for links.  While theorems about the conditions under which classes {\it are} simple generalize straightforwardly from the knot setting to links, the construction of various types of {\it non-simple} classes is richer in the setting of links.

To establish the above machinery, we consider a more general problem.  Let $(M,V)$ denote a compact, oriented smooth 3-manifold $M$ equipped with a nowhere-vanishing smooth vector field $V$.  When $M$ is equipped with a co-oriented contact structure, we will take $V$ to be a co-orienting vector field for $\xi$.  Next we consider a version of the simplicity problem for links in $(M,V)$ which are everywhere transverse to $V$; that is, the tangent vector to the link and the vector $V$ at the corresponding point span a 2-plane. Such links are called {\it $V$-transverse}. When $V$ is a co-orienting vector field for a contact structure, $V$-transverse links are also called {\it pseudo-Legendrian}.  Pseudo-Legendrian links were introduced by Benedetti and Petronio \cite{benedetti2001combed, benedetti2001reidemeister}.  Each $V$-transverse link is equipped with a natural framing.  In the $V$-transverse category, a framed isotopy class of links is {\it simple} if any two $V$-transverse links in that isotopy class which are homotopic as $V$-transverse immersions are isotopic as $V$-transverse links. When that is not the case, the framed isotopy class is again called {\it non-simple}.  When the vector field $V$ is a co-orienting vector field for a contact structure $\xi$ on $M$, we show that if a framed isotopy class in $M$ is non-simple in the $V$-transverse category, then it is non-simple in the Legendrian category.  

\begin{rem}
    All links in this paper come with an ordering on their components.  All homotopies, link homotopies, and isotopies of links are assumed to preserve this ordering, i.e., they do not permute the components. 
\end{rem}

Our next task is to consider refinements of this notion of simplicity in both the $V$-transverse and Legendrian categories. We use the term {\it component-wise} isotopy to refer to a generic homotopy of a link in which double points occur only between distinct components of the link (i.e., for each component of the link, the restriction of the homotopy to that component is an isotopy). We use the term {\it link homotopy} to refer to a generic homotopy of a link in which {\it no} double points occur between distinct components of the link.  Component-wise isotopies and link-homotopies can be defined in the natural way in both the $V$-transverse and Legendrian categories.

The following is a summary of our main theorems. The details will be given in Section \ref{examples.sec}.
\begin{thm}\label{intro.thm} Let $M$ be an $S^1$-bundle over an oriented surface of genus at least 2. For each $k\in\mathbb{Z}$ we construct a nowhere-zero vectorfield $V_k$ on $M$, and $V_k$-transverse links in $(M, V_k)$, which are framed isotopic and homotopic as $V$-transverse multi-curves (i.e., have the same generalized classical invariants), and
    \begin{itemize}
        \item are component-wise isotopic as $V_k$-transverse links, but are not link-homotopic and hence not isotopic as $V_k$-transverse links, or
        \item are link-homotopic as $V_k$-transverse links, but are not component-wise isotopic and hence not isotopic as $V_k$-transverse links.
        
    \end{itemize}
\end{thm}

Our first examples of links with the properties in the two bullet points above occur distinct smooth isotopy classes.  Surprisingly, by choosing different $V$-transverse stabilizations of the link components (a notion introduced in Section~\ref{stab.sec}) we then construct distinct pairs of $V$-transverse links, one pair satisfying each of the criteria bullet point above, in the {\it same} smooth isotopy class.

 As a corollary we prove the following:

\begin{cor}\label{leg_summary.cor}
Let $(M, \xi_k)$ be an overtwisted manifold corresponding to the choice of $(M,V_k)$ in Theorem~\ref{intro.thm}. There exists Legendrian links in $(M,\xi_k)$ which are framed isotopic and homotopic as Legendrian multicurves (i.e., have the same generalized classical invariants), and are
\begin{itemize}
    \item component-wise Legendrian isotopic, but are not Legendrian link-homotopic, and hence not Legendrian isotopic as links, or
    \item  are link-homotopic as Legendrian links, but are not component-wise isotopic and hence not isotopic as Legendrian links.

\end{itemize}
 
\end{cor}
In Corollary \ref{leg_summary.cor}, the components of the link are loose in $(M,\xi_k)$. Again, our examples show that both of these non-simplicity phenomena can occur in a single smooth isotopy class, by controlling the number of stabilizations of each of the link components. Note that we construct an infinite family of Legendrian links which are component-wise Legendrian isotopic but not Legendrian isotopic as links. This phenomena has been observed previously in tight manifolds \cite{jordan2006generating, chatterjee2025cable}.

Finally,  we state the following theorem, which gives sufficient conditions for a framed isotopy class to be simple in the $V$-transverse category.  This theorem was proven for isotopy classes of knots in \cite{CC}, and generalizes immediately to links. It highlights the fact that the above examples must sit in overtwisted contact manifolds. The key ingredient in the proof of statement (3) is that the Euler class of a tight contact structure vanishes on smooth mappings of tori, as shown in \cite[Corollary 3.10]{chernov2004relative}; in the embedded case this is due to Eliashberg \cite{eliashberg1989classification}.

\begin{thm}[\cite{CC}]\label{thm:framed_simple}
    Let $V$ be a nowhere-zero vector field on an oriented $3$-manifold $M$ satisfying one of the following three conditions:
    \begin{enumerate}
        \item The Euler class $e_{{V}^\perp}\in H ^2(M;\mathbb{Z})$ is a torsion element, or in particular, if $e_{{V}^\perp}=0.$
        \item The manifold $M$ is closed, irreducible and atoroidal.
        \item $V$ is a co-orienting vector field of a contact structure $\xi$ such that $(M,\xi)$ is tight, or more generally, such that $(M,\xi)$ is a covering of a tight contact manifold.
    \end{enumerate}
    Then every framed isotopy class  of $n$-component multi-curves in $M$ is $V$-transversely simple.
\end{thm}

\section{Background on $V$-transverse links}\label{background.sec}

All 3-manifolds $M$ are smooth, closed, connected, and oriented, and equipped with an auxilliary Riemannian metric.  All spaces of knots and curves in $M$ are equipped with $C^\infty$-topology. Isotopy and homotopy classes of links and multi-curves are defined to be the path components of the spaces of links and multi-curves in $M$, respectively.

\begin{defn}
    A {\it framed multi-curve} $C=(C_1, C_2,\cdots, C_k )$ in an oriented $3$-manifold $M$ is a framed immersion of a  disjoint union of $S^1$'s, i.e each $C_i\colon S^1\rightarrow M$ is an immersion together with a non-vanishing section of its (induced) normal bundle. A {\it framed link} is a framed multi-curve that is also an embedding. 
\end{defn}

\begin{defn}
   Let $V$ be a nowhere-vanishing vector field on an oriented 3-manifold $M$. A {\it $V$-transverse multi-curve} is a multi-curve $C=(C_1,C_2,\cdots C_k)$ in $M$ such that each $C_i$ is an immersion of $S^1$ into $M$ and $C_i'(t) $ and $V_{C_i(t)}$ span a $2$-plane for all $t\in S^1.$ A $V$-transverse link is a $V$-transverse multi-curve that is also an embedding.
\end{defn}
Each component of a $V$-transverse multi-curve has a natural framing given by the orthogonal projection of $V_{C_i(t)}$ to the normal bundle of the curve $C_i(t)$.  A framed homotopy (resp. isotopy) of $V$-transverse multi-curves (resp. links) is a framed homotopy (resp. isotopy) between the $V$-transverse multi-curves (resp. links) equipped with this natural framing.  The intermediate multi-curves (resp. links) need not be $V$-transverse.

When $V$ is a co-orienting vector field for a contact structure on $M$, we call a $V$-transverse multi-curve or link {\it pseudo-Legendrian.}

\subsubsection{Generalized Classical Invariants and the Botany Problem}

Classically, the botany problem in Legendrian knot theory asks for a classfication of all Legendrian knot types in a given smooth isotopy class with a given Thurston-Bennequin and rotation number. For links, one can instead classify all Legendrian link types in a given smooth isotopy class whose components have given sequences of Thurston-Bennequin and rotation numbers. Of course, this formulation of the botany problem makes sense only when both the Thurston-Bennequin and rotation numbers are defined for all link components. As mentioned in the introduction, the Thurston-Bennequin number of a Legendrian knot is defined only when the knot is null-homologous, as the self-linking number of the knot with respect to its contact framing.  The rotation number of a knot is defined either for a null-homologous knot, in which case the rotation number can be interpreted as a relative Euler class evaluated on a choice of Seifert surface, or for a knot in a parallelizable contact structure $\xi$, in which case the rotation number can be viewed as the degree of the mapping $K\mapsto K'$ with respect to a choice of trivialization of $\xi$. 

The same challenge arises when formulating the botany problem for $V$-transverse links.  The self-linking number of each component of a $V$-transverse link (computed with respect to the natural framing of the $V$-transverse link) are defined only when the component is null-homologous. When the 2-plane bundle $V^\perp$ is parallelizable, one can define the rotation number $\text{rot}(K)$ of each component $K$ by orthogonally projecting $K'$ to $V^\perp$ and computing the degree of the mapping $K\mapsto \text{proj}_{V^\perp}K'$ with respect to a choice of trivialization of $V^\perp$.

The next three propositions, whose proofs are straightforward, motivate the use of the framed isotopy class and $V$-transverse (resp. Legendrian) homotopy class as the generalized classical invariants that should be used in the generalized $V$-transverse (resp. Legendrian) botany problem in arbitrary 3-manifolds, when the link components are not null-homologous, or the contact structure is not parallelizable.  (Note that we do not use these facts later in the paper.)

\begin{prop}
    Two framed  links in $\mathbb{R}^3$  are isotopic as framed links if and only if they are smoothly isotopic and their respective components have the same self-linking number.
\end{prop}

\begin{prop} Suppose $V$ is a nowhere-zero vector field on the 3-manifold $M$, such that $V^\perp$ is parallelizable.
    Let $C_1$, $C_2$ be $V$-transverse curves in the same homotopy class of immersed curves in $M$. Then $C_1$ and $C_2$ are homotopic as $V$-transverse curves if and only if $\text{rot}(C_1)=\text{rot}(C_2)$. 
\end{prop}

\begin{prop} Suppose $\xi$ is a parallelizable, co-oriented contact structure on an oriented 3-manifold $M$.   Let $C_1$, $C_2$ be Legendrian curves in the same homotopy class of immersed curves in $M$. Then $C_1$ and $C_2$ are homotopic as Legendrian curves if and only if $\text{rot}(C_1)=\text{rot}(C_2)$. 
\end{prop}

The above discussion shows that the appropriate generalization of the classical Botany Problem in the $V$-transverse category is:

 {\it Classify all $V$-transverse isotopy classes of $V$-transverse links in a specified isotopy class of framed links and specified homotopy class of $V$-transverse multi-curves.}

Similarly, the generalization of the classical Botany Problem for Legendrian links in arbitrary contact manifolds is:

 {\it Classify all Legendrian isotopy classes of Legendrian links a specified isotopy class of framed links and specified homotopy class of Legendrian immersions.}

\subsection{Kinks and stabilization}\label{stab.sec}
We next define a stabilization operation for $V$-transverse knots/links, and compare it with the well-known stabilization operation in the Legendrian category. Let $K$ be a component of a $V$-transverse multicurve in $M$. We consider a coordinate chart $\phi:U\rightarrow M$ such that $U$ contains an unknotted arc of $K$ and $V=\phi^{-1}_*(\frac{\partial}{\partial z})$. Then the $V-$transverse curve $K^i$ is obtained from $K$ by adding $i$ pairs of kink as shown on the top of Figure \ref{kink.fig} if $i>0$. If $i<0$, then we add $|i|$ pairs as shown in the bottom of Figure \ref{kink.fig} and denote the result by $K^{-i}$.

 \begin{prop}\cite{CC}
     The $V$-transverse knots $K^i$ are all isotopic are framed knots. 
 \end{prop}

 The $V$-transverse stabilization operation is not the same as stabilization in the Legendrian category.  The operations are related as follows. The stabilization of a Legendrian link component $L$ occurs in a local chart contactomorphic to $(\mathbb{R}^3,\xi_{std})$ where $\xi_{std}=\ker(dz-ydx)$. Such a chart always exists by Darboux's theorem. Suppose that $L_{i,j}$  denotes the Legendrian knot obtained from $L$ by adding $i$ positive and $j$ negative stabilizations.  See Figure~\ref{fig:stab}. As shown in \cite{CC}, $L_{i,0}$ and $(L_{0,i})^i$ (the superscript denoting $V$-transverse stabilization) are isotopic as $V$-transverse knots where $V=\phi_*^{-1}(\frac{\partial}{\partial z})$; the isotopy can be seen in the Lagrangian projection. Since $\frac{\partial}{\partial z}$ is a co-orienting vector field for $\xi_{std}$, this $V$-transverse isotopy is also a pseudo-Legendrian isotopy.

\begin{figure}
\labellist
\small\hair 2pt
\pinlabel {$K^{i}$} at 10 250
\pinlabel {$K^{-i}$} at 20 35

\endlabellist
    \centering
    \includegraphics[scale=0.4]{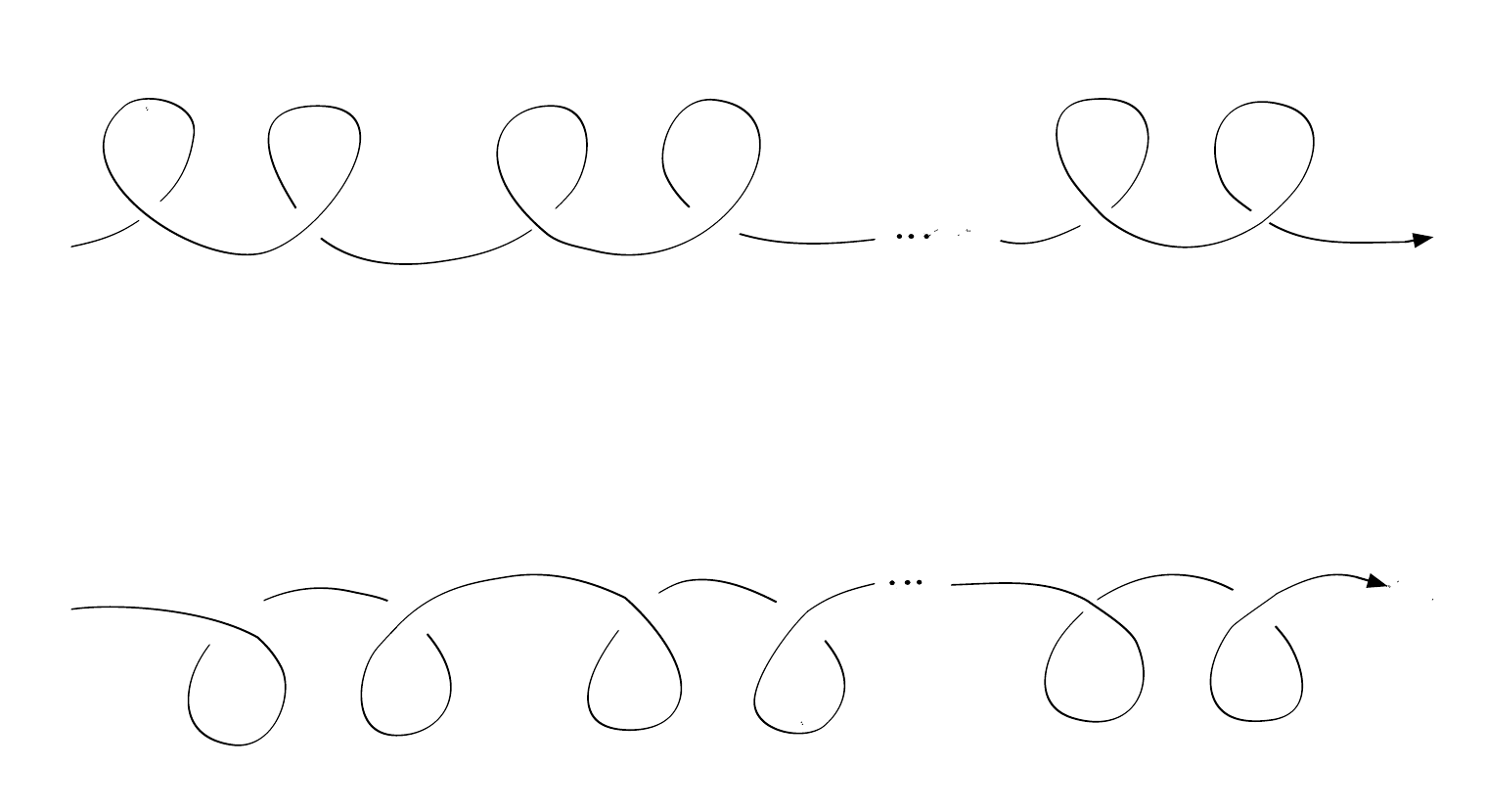}
    \caption{$V$-transverse stabilizations $K^i$ and $K^{-i}$. In this local chart, $\frac{\partial}{\partial z}$ points out of the page.}
    \label{kink.fig}
\end{figure}

\begin{figure}
\labellist
\small\hair 2pt
\pinlabel {$L$} at 450 480
\pinlabel{$L_{i,0}$} at 220 200
\pinlabel {$L_{0,i}$} at 680 200

\endlabellist
    \centering
    \includegraphics[scale=0.4]{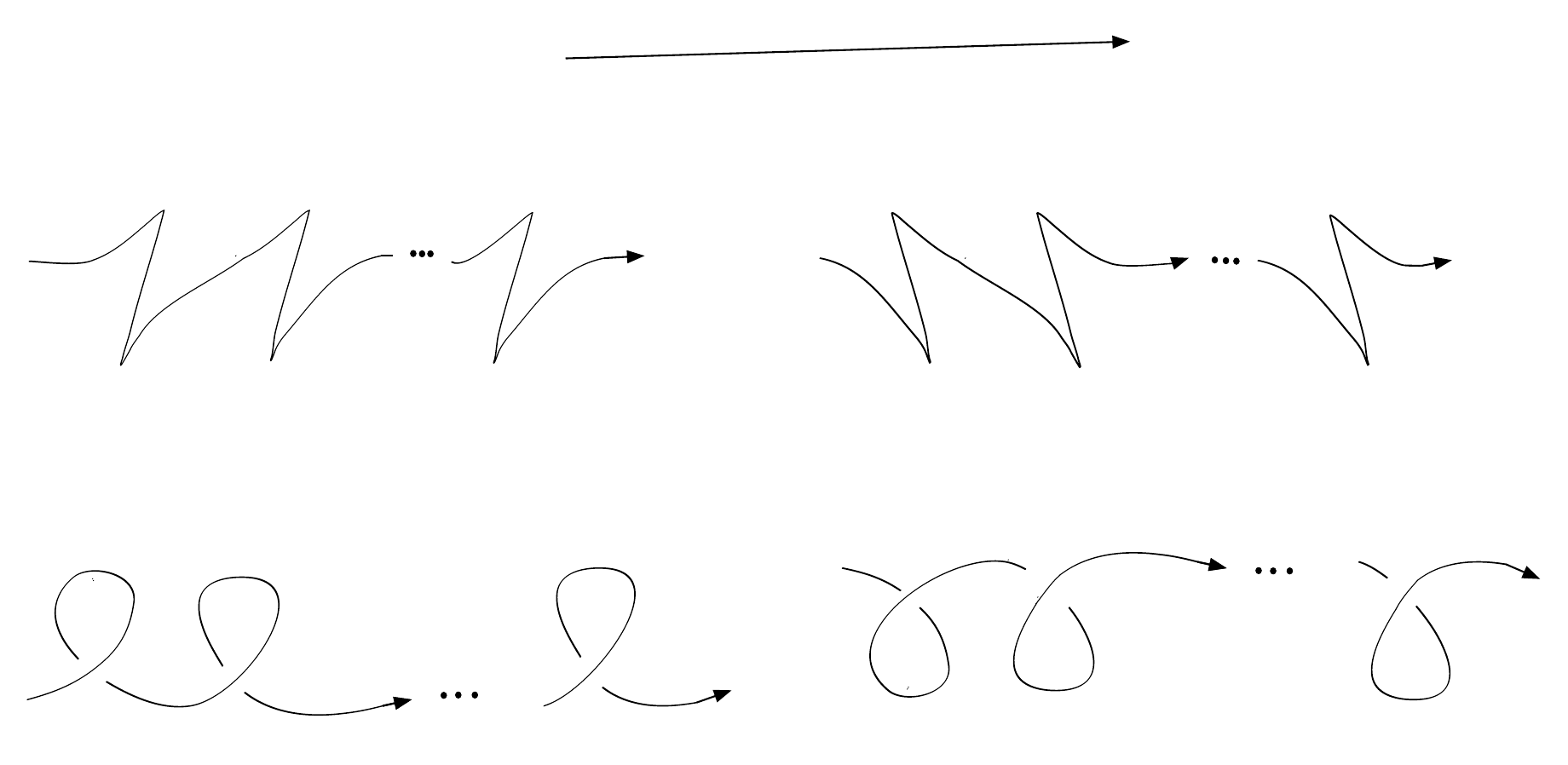}
    \caption{The Legendrian link component $L_{i,0}$ and $L_{0,i}$ in the front and Lagrangian projections. In the Lagrangian projection, $\frac{\partial}{\partial z}$ points out of the page.}
    \label{fig:stab}
\end{figure}

\section{Spaces of Framed and $V$-Transverse Multicurves}

\subsection{Approximating by $V$-transverse homotopies and isotopies} We recall two propositions from \cite{CC} for knots and generalize directly to links. To simplify notation, we write the result for knots and 2-component links, but the links can have any number of components.

\begin{prop}[\cite{CC}, Lemma 3.4]\label{approximate_framed.prop}
    Any framed isotopy from $K$ (resp. $K_1\sqcup K_2$) to $K'$ (resp. $K_1'\sqcup K_2'$) can be $C^0$-approximated by a $V$-transverse isotopy. In particular, $K$ (resp. $K_1'\sqcup K_2'$) is $V$-transversely isotopic to $K^i$ (resp. $K_1^{i_1}\sqcup K_2^{i_2}$) for some $i$ (resp. $i_1,i_2$) $\in \mathbb{Z}$.
\end{prop}

Moreover, $K$ is $V$-transversely homotopic to $K^i$ if and only if there is a framed self-homotopy of $K$ such that the Euler class of $V^\perp$ evaluated on the corresponding map of $S^1\times S^1\rightarrow M$ is $2i$, as we now recall. Let $\mathcal{C}$ denote a connected component of the space of immersed curves in $M$, containing a $V$-transverse knot $K$.  Next we recall the map $h_V:\pi_1(\mathcal{C},K)\rightarrow \mathbb{Z}$ from \cite{CC} which measures the failure of the $V$-transverse knots $K$ and $K^i$ to be $V$-transversely homotopic.  Let $[\alpha]\in  \pi_1(\mathcal{C},K)$ be a homotopy class, with $\alpha:S^1\rightarrow \mathcal{C}$.  This map can be regarded as a map of a torus $S^1\times S^1\rightarrow M$, also denoted $\alpha$. Recall that $e_{V^\perp}\in H^2(M;\mathbb{Z})$ denotes the Euler class of the 2-plane bundle $V^\perp$ on $M$. Then define

$$h_V([\alpha]):=\frac{1}{2}e_{V^\perp}(\alpha_*[S^1\times S^1]).$$

\begin{prop}\cite[Lemma 4.4]{CC}\label{hV_torus} Let $K$ be a $V$-transverse knot in $M$, and let $\mathcal{C}$  be the connected component of the space of immersed curves containing $K$.  Let $h_V$ be the kink-cancelling homomorphism defined above.  Then a framed self-homotopy representing a class $\alpha\in \pi_1(\mathcal{C},K)$ can be $C^0$-approximated by a $V$-transverse homotopy from $K$ to $K^i$ if and only if $h_V[\alpha]=i$. In particular, $\alpha$ can be a approximated by a $V$-transverse self-homotopy of $K$ if and only if $[\alpha]\in\text{ker }h_V.$
\end{prop}

In the next section, we will find a set of generators for $\pi_1(\mathcal{C},K)$ when $M$ is an $S^1$-bundle over an oriented surface of genus $\geq 2$.  We will then apply Proposition ~\ref{hV_torus} to find a corresponding generators for $\pi_1(\mathcal{C}_V,K)$, where $\mathcal{C}_V$ is the connected component of the space of $V$-transverse curves containing a given $V$-transverse curve $K$.

\begin{figure}
    \centering
\labellist
\small\hair 2 pt

\pinlabel $(+)$ at 130 250
\pinlabel $(-)$ at 430 250

   \endlabellist
    \includegraphics[width=0.5\linewidth]{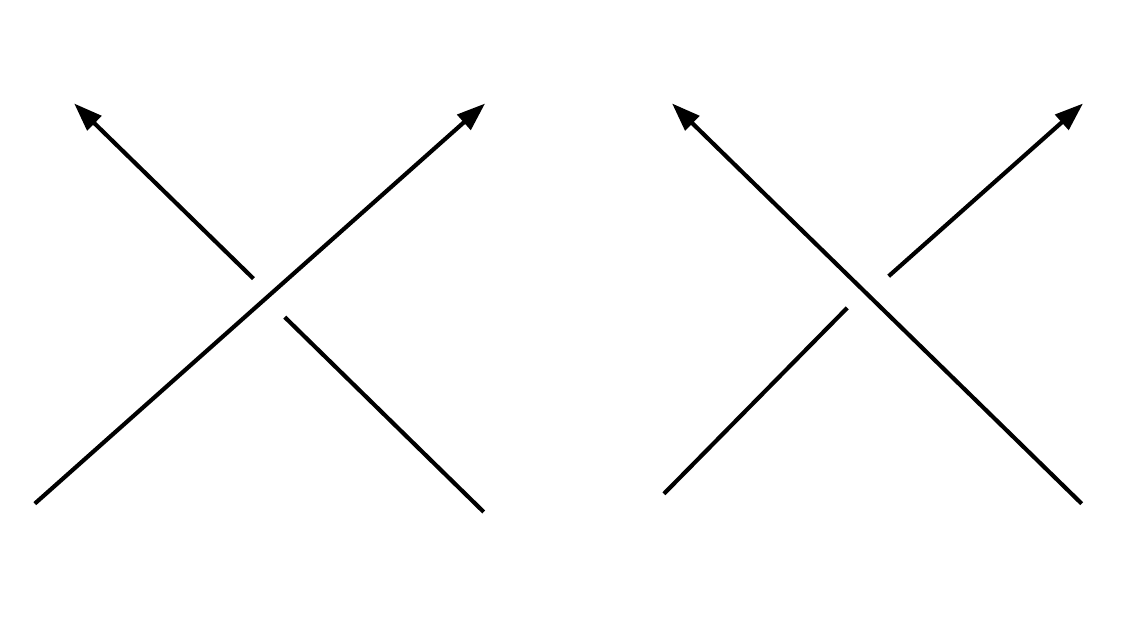}
    \caption{positive and negative crossing.}
    \label{fig:crossing}
\end{figure}

\begin{figure}
    \centering
\labellist
\small\hair 2 pt

\pinlabel $(1,-1)$ at 280 350
\pinlabel $(-1,1)$ at 560 350
\pinlabel $(1,1)$ at 280 190
\pinlabel $(-1,-1)$ at 580 190

   \endlabellist
    \includegraphics[width=0.5\linewidth]{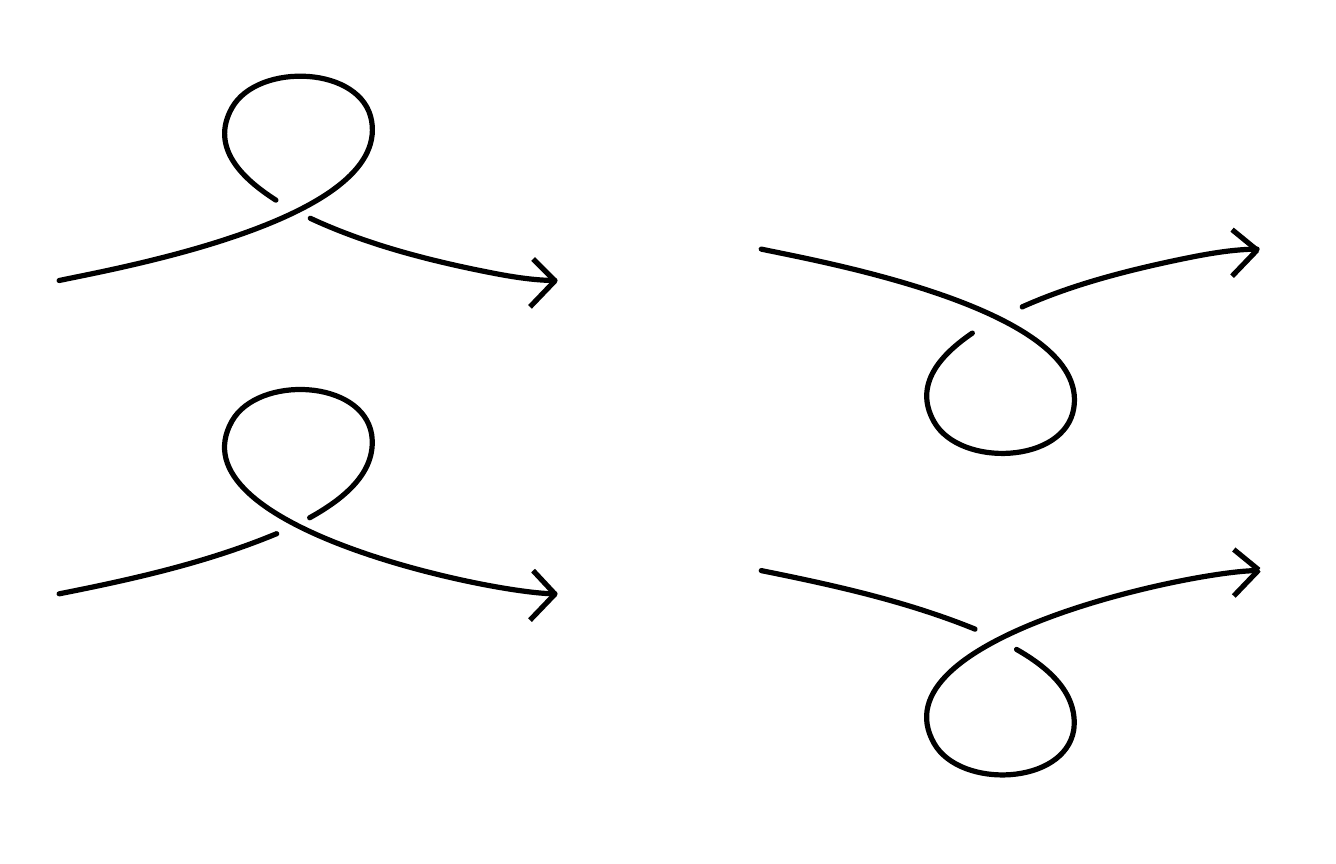}
    \caption{Four different types of kinks. The coordinates $(a,b)$ represent their contribution $a$ to a local rotation number and $b$ to a local writhe number of the diagram. Pairs of kinks with opposite local rotation number and opposite local writhe
number can be created or cancelled by a $V$-transverse isotopy as shown in \cite{CC}.}
    \label{fig:kinks}
\end{figure}

\begin{figure}
    \centering

    \includegraphics[width=0.5\linewidth]{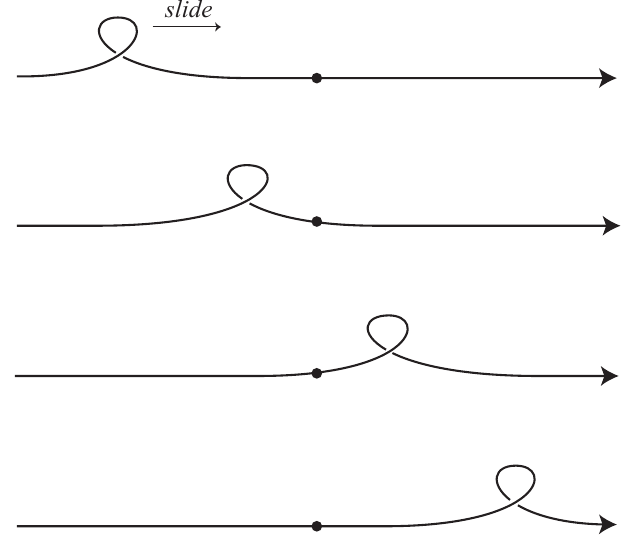}
    \caption{A loop in the space of $V$-transverse immersions corresponding to the $S^1$-fiber of $E_V$, created by sliding a small kink around a knot until it returns to its starting point.}
    \label{fig:kinkslide}
\end{figure}

\begin{figure}
    \centering
    \includegraphics[width=0.5\linewidth]{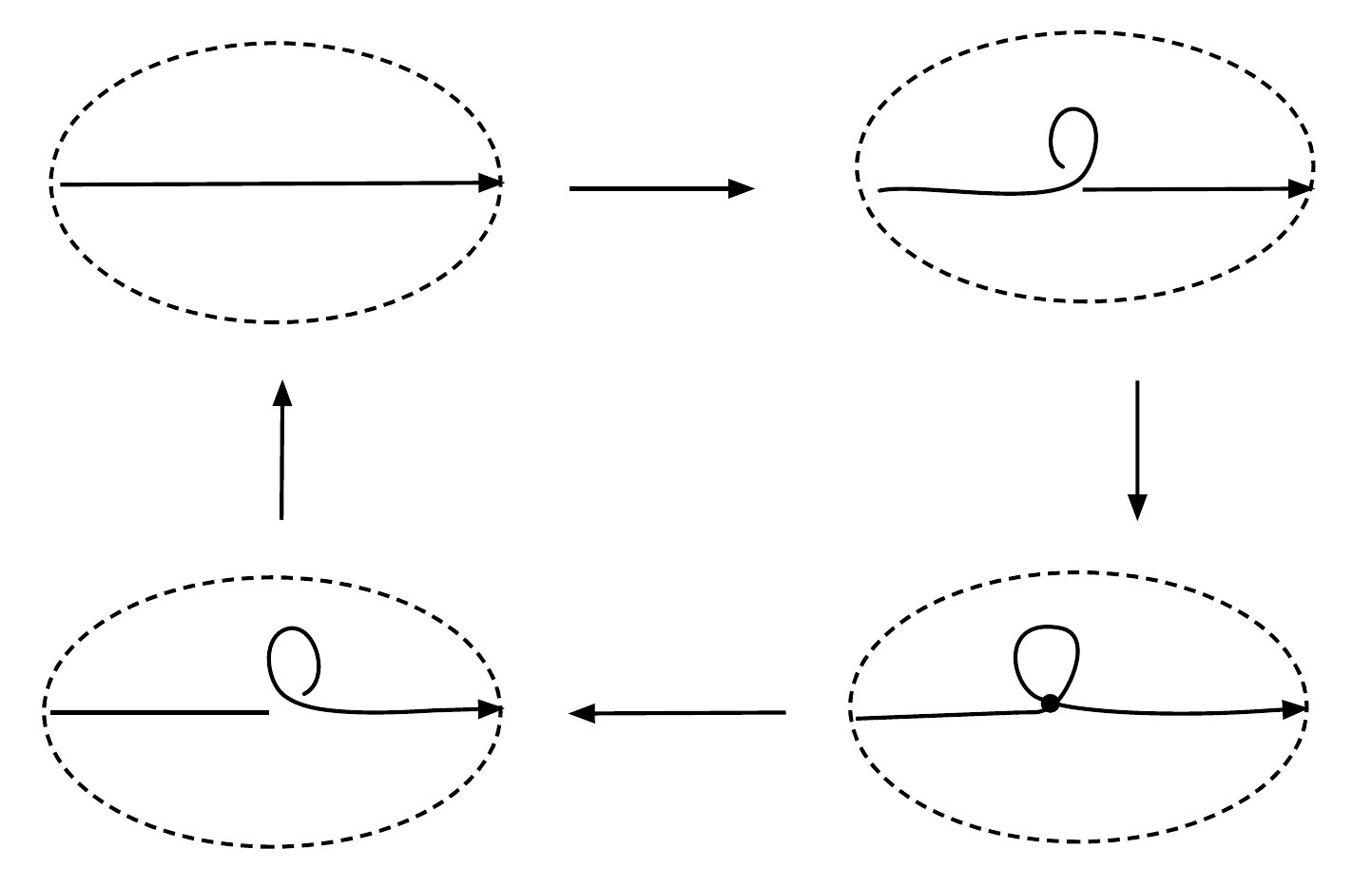}
    \caption{The loop $\gamma_s$.}
    \label{fig:loop}
\end{figure}

\subsection{Generators of $\pi_1$ of the spaces of immersed, framed, and $V$-transverse curves in an $S^1$-bundle over a surface}\label{generators.sec}

Recall that $\mathcal{C}$ denotes a path-component of the space of immersed curves in $M$.  We begin by describing two nontrivial elements of $\pi_1(\mathcal{C},K)$ for some immersed curve $K$ in an arbitrary 3-manifold $M$. The proof that these elements are indeed nontrivial is contained in the proof of Lemma~\ref{lem:Vtranscurveclassification}.
\begin{itemize}
\item Define $[\gamma_s]\in \pi_1(\mathcal{C},K)$ to be the class of a loop $\gamma_s$ which creates a small kink on an arc of $K$, passes through the corresponding double point to change the kink to one of opposite local writhe number, and then removes the kink. See Figure~\ref{fig:loop}. 

\item Given any immersed curve $K$, let $\gamma_{rot}$ be the loop in $\mathcal{C}$ which rotates $K$ along itself, induced by a full rotation of the parametrizing circle. 
\end{itemize}

Now we turn to the specific case where $M$ is an $S^1$-bundle over an oriented surface $F$, which will be the case in all of our examples in Section~\ref{examples.sec}, and define two additional elements of $\pi_1(\mathcal{C},K)$ in this setting. (In Lemma~\ref{lem:Vtranscurveclassification}, which is used for all the examples in Section~\ref{examples.sec}, we also assume the genus of $F$ is at least two.)

\begin{itemize}

\item If $K$ is a curve that lies in a torus neighborhood of an $S^1$-fiber $f$ of $M$ over a point $p\in F$ (and hence homotopic to a power $f^l$ of the fiber), and $\rho\in \pi_1(F,p)$, we let $\gamma_{\rho}$ be a loop in $\mathcal{C}$ that drags the torus neighborhood of $f$, and hence $K$, around a loop projecting to $\rho$. Note that $\gamma_{\rho}$ is only well-defined up to powers of $\gamma_{\text{rot}}$.  Given a basepoint $t_0\in S^1$, a specific choice of $\gamma_\rho$ can be made by specifying the trace of this basepoint as an element $\pi_1(M,K(t_0))$.

\item Finally, suppose $K$ is not homotopic to a power of the $S^1$ fiber $f$.  In this case we define $\gamma_{\text{fib}}$ to be a loop which simultaneously rotates each point of $K$ along the $S^1$-fiber of $M$ containing that point, in the direction of the orientation of the fiber. 
\end{itemize}

When $K$ is equipped with a framing, we can consider the connected component of the space of framed curves $\mathcal{C}_f$ containing $K$.  Let $\mathcal{C}_f'$ be a quotient of  $\mathcal{C}_f$, obtained by identifying pairs of framed curves whose underlying immersed curves are equal and whose framings are homotopic as nonzero sections of the induced normal bundle of $K$.  There is a covering map $j: \mathcal{C}_f'\rightarrow \mathcal{C}$ with discrete fiber \cite[Section 3.1]{chernov2005framed}.

When $K$ is framed, each of the above loops $\alpha$, with $[\alpha]\in \pi_1(\mathcal{C},K)$, then lifts to a path $\tilde{\alpha}$ in $\mathcal{C}_f'$ based at $K$.  Given any choice of framing on the basepoint curve $K$, the loop $\gamma_s$ (which is a self-homotopy, but not a self-isotopy of $K$) induces a full twist of the framing of $K$.  Hence $\gamma_s$ lifts to a path, but not a loop, in $\mathcal{C}_f'$. The main result of \cite{cahn2014number}, in the case where the ambient manifold is orientable, can be phrased as follows: If $\alpha$ is a self-isotopy of $K$ in $\mathcal{C}$, then the lift $\tilde\alpha$ is a loop in $\mathcal{C}_f$ if and only if $K$ does not intersect a non-separating 2-sphere in $M$.  When $M$ is an $S^1$-bundle over an orientable surface of genus $\geq 1$, no such $2$-sphere exists.  Therefore, $\gamma_\rho$, $\gamma_\text{rot}$, and $\gamma_\text{fib}$ all lift to loops in $\mathcal{C}_f$.  

When $K$ is not just framed, but $V$-transverse, we can then apply Proposition~\ref{hV_torus} to determine when each of these loops can be realized by a $V$-transverse self-homotopy (and in this case, self-isotopy) of $K$, and hence represent classes in $\pi_1(\mathcal{C}_V,K)$. Let $\iota:\mathcal{C}_V\rightarrow \mathcal{C}$ denote the inclusion of the component of the space of $V$-transverse curves containing $K$ into the component of the space of immersed curves containing $K$.

The final ingredient in the classification of $V$-transverse curves is a loop in $\pi_1(\mathcal{C}_V,K )$, shown in Figure~\ref{fig:kinkslide}, which is in $\text{ker }\iota_*$.   The loop $\gamma_{\text{kink}}$ is achieved by sliding one kink around the image of $K$, after creating a pair of opposite kinks (as shown in Figure \ref{fig:kinks}) if necessary, in order to create a local kink to slide.

We can now classify all loops in the space $V$-transverse curves in $M$.

\begin{lem}\label{lem:Vtranscurveclassification}  Let $K$ be a $V$-transverse knot in $M$, an $S^1$-bundle over an oriented surface $F$ of genus at least two.  Let $[d]\in H_1(M;\mathbb{Z})$ be the Poincar\'e dual of $e_{V^\perp}\in H^2(M;\mathbb{Z}).$  
\begin{enumerate}
\item If $K$ is not homotopic to a power of the vertical fiber, then every element of $\pi_1(\mathcal{C}_V,K)$ is of the form
$$\gamma_{\text{rot}}^a\gamma_{\text{fib}}^b\gamma_{\text{kink}}^c,$$
where $a, b\in \mathbb{Z}$, and the evaluation of $e_{V^\perp}$ on the torus corresponding to $\gamma_{\text{fib}}$ is $0$.
    \item If $K$ is homotopic to a power of the $S^1$-fiber of $M$, then every element of $\pi_1(\mathcal{C}_V,K)$ is of the form
$$ \phi\gamma_{\text{rot}}^a\gamma_{\rho}\gamma_{\text{kink}}^c\phi^{-1},$$
where $a, b\in \mathbb{Z}$, $\phi$ is a homotopy that carries $K$ into a torus neighborhood of a vertical fiber, and the evaluation of $e_{V^\perp}$ on the torus corresponding to $\gamma_{\rho}$ is $0$.

\end{enumerate}
    
\end{lem}

\begin{rem}
    The final conditions on the Euler class of $V^\perp$ can be checked by counting intersection numbers of curves on $F$ with the projection to $F$ of the dual of $e_{V^\perp}$.
\end{rem}
\begin{proof}
Our strategy is to first classify all loops in a path-component $\mathcal{C}$ of the space of immersed curves in $M$, and from there derive a classification of loops in a path-component $\mathcal{C}_V$ of the space of $V$-transverse curves in $M.$
We first recall some facts about $\pi_1(\mathcal{C},K)$.  The $h$-principle for the space of immersions of $S^1\rightarrow M$ states that $\text{Imm}(S^1,M)$ is weak homotopy equivalent to the free loop space $\Omega STM$, where $STM$ is the spherical tangent bundle of $M$ \cite{gromov1986partial}. A map $\mathcal{C}\hookrightarrow \text{Imm}(S^1,M) \rightarrow \Omega STM$ is given by lifting an immersed curve $C$ via its unit tangent vector to a curve $\tilde C$ in $STM$.  Given a representative $a$ of a class $\alpha\in \pi_1(\mathcal{C},K)$, which we view as a map $a:S^1\times S^1\rightarrow M$, we similarly lift $a$ to a map $\tilde a:S^1\times S^1\rightarrow STM$, representing a class $\tilde\alpha\in\pi_1(\Omega STM,\tilde K)$. 
Choose a basepoint $t_0\in S^1$.  Given a loop $\tilde a$ representing $\tilde \alpha\in \pi_1(\Omega STM,\tilde{K})$, we let $t(\tilde \alpha)\in \pi_1(STM,\tilde{K}(t_0))$ denote the class containing the trace of the basepoint $\tilde{K}(t_0)$ under $\tilde{a}$. The obstruction for maps $\tilde a_1$ and $\tilde a_2:S^1\times S^1\rightarrow STM$, with $\tilde a_1(t\times 0)=\tilde a_2(t\times 0)=\tilde K (t)$, and with the same trace, to be homotopic, is an element of $\pi_2(STM).$ Since $M$ is an oriented 3-manifold, it is parallelizable, so $\pi_2(STM)=\pi_2(S^2)\oplus \pi_2(M)=\mathbb{Z}\oplus \pi_2(M)$. Note that in our setting, where $M$ is an $S^1$ bundle over an oriented surface of genus at least 2, $\pi_2(M)=0$. The generator of the $\pi_2(S^2)=\mathbb{Z}$ summand corresponds to the loop $\gamma_s\in \pi_1(\mathcal{C},\tilde K)$ defined in Section ~\ref{generators.sec}.

It follows that if $\alpha_1,\alpha_2\in \pi_1(\mathcal{C}, K)$ have the same trace (where the trace is considered as an element of $\pi_1(STM)$), then $\alpha_2=\alpha_1 \gamma_s^m$; see \cite[Lemma 3.6]{chernov2005framed}.  We next show that for any $\alpha\in \pi_1(\mathcal{C},K)$, there exists $i\in \mathbb{Z}$ such that
\begin{equation}\label{immersed_curves.eq}
\alpha^i=\begin{cases}
 \gamma_{\text{rot}}^a \gamma_{\text{fib}}^b\gamma_s^{m} &\text{ if } [K]\neq [f^j]\in\pi_1(M)\\
  \phi\gamma_{\text{rot}}^a  \gamma_\rho  \gamma_s^m \phi^{-1}& \text{ if } [K]=[f^j]\in\pi_1(M)\\
\end{cases}
\end{equation}
where $f$ is the $S^1$ fiber of $M.$

  Given $\alpha\in \pi_1(\mathcal{C},K)$, consider its lift $\tilde \alpha\in \pi_1(\Omega STM,\tilde{K})$, and the subgroup $\tilde{A}=\langle [\tilde{K}],[\text{tr}(\tilde{\alpha})]\rangle \leq \pi_1(STM, \tilde{K}(t_0)).$  Note that since $M$ is parallelizable, $\text{pr}_*$ induces an isomorphism $\text{pr}_*:\pi_1(STM, \tilde{K}(t_0))\rightarrow \pi(M,K(t_0))$, so we consider the image $A=\text{pr}_*(\tilde A)$ of $\tilde A$, namely, the group  $A=\langle [{K}],[\text{tr}({\alpha})]\rangle  \leq \pi_1(M,K(t_0))$.  As $\tilde \alpha:S^1\times S^1\rightarrow STM$, $A$ is an abelian subgroup of $\pi_1(M,K(t_0))$.  \\
Let $\text{pr}':M\rightarrow F$ be the bundle map realizing $M$ as an $S^1$-fibration over $F$, with $S^1$-fiber $f$.  We now have an abelian subgroup $\text{pr}_*'(A)\leq \pi_1(F,p)$, which, since the genus of $F$ is at least two, lies in a unique infinite cyclic subgroup \cite[Theorem 3.2]{doCarmo1992}.  Choose a generator $s$ of this subgroup and choose $\sigma\in \pi_1(M,K(t_0))$ such that $\text{pr}_*'(\sigma)=s$.  The long exact sequence of the fibration $\text{pr}':M\rightarrow F$ yields that two such choices of $\sigma$ differ by a power of $[f]$, and for any such choice of $\sigma$, $A=\langle [K],[\text{tr}(\alpha)]\rangle\leq \langle [f],\sigma\rangle.$ Hence we can write $K=\sigma^i[f]^j$ and $\text{tr}(\alpha)=\sigma^k[f]^l.$\\
Suppose $i\neq 0$, putting us in the case $[K]\neq [f^j]\in\pi_1(M).$
We have
$$\text{tr}(\alpha^i)=\text{tr}(\alpha)^i=\sigma^{ki}[f]^{li}=(\sigma^i)^k([f]^j)^{il-jk}.$$
Then set $a=l$ and $b=il-jk.$ This gives $[\text{tr}(\alpha)]^i=[K]^a[f]^b$, which is also the trace of $\gamma_\text{rot}^a\gamma_\text{fib}^b.$ Hence $\alpha^i=\gamma_\text{rot}^a\gamma_\text{fib}^b\gamma_s^m$ for some $m\in\mathbb{Z}.$\\
Now suppose $i=0$, so that $[K]=[f^j]\in\pi_1(M).$ By pre-composing with a homotopy $\phi$ of $K$, and post-composing by $\phi^{-1}$, we may assume $K$ lies in a torus neighborhood of a vertical $S^1$-fiber of $M$. Since $K$ lies in a torus neighborhood of an $S^1$-fiber, we can apply the homotopy $\gamma_{\rho'}$ to $K$, with $\rho'=\text{pr}_*'(\sigma)$.  Recall that $\gamma_{\rho'}$ is well-defined up to powers of $\gamma_{\text{rot}},$ and  $\text{tr}(\gamma_{\text{rot}})=[K]=[f]^j.$ For $\rho=\text{pr}_*(\sigma)$, we can write $\text{tr}(\gamma_{\rho'})=\sigma[f]^{jn}$ for some $n\in\mathbb{Z}$.  Let $\sigma'=\sigma[f]^{jn}$, so that $\text{tr}(\gamma_{\rho})=\sigma'$. \\
We now consider
$$\text{tr}(\alpha^{j})=\text{tr}(\alpha)^{j}=\sigma^{kj}[f]^{lj}=(\sigma')^{kj}([f]^{-jn})^{kj}([f]^j)^{l}=(\sigma')^{kj}([f]^j)^{l-jkn}.$$
Then set $a=l-jkn$ and $c=kj$.  This gives $[\text{tr}(\alpha)]^j=(\sigma')^c[K]^a$, which is also the trace of $(\gamma_{\text{rot}})^a\gamma_{\rho'}^{c}$ when $\rho'=\sigma'$. Hence $\alpha^j=(\gamma_{\text{rot}})^a\gamma_{\rho'}^{c}\gamma_s^m$ for some $m\in \mathbb{Z}$.  By taking $\rho=(\rho')^c$, we can choose $\gamma_\rho$ so that  $\alpha^j=(\gamma_{\text{rot}})^a\gamma_{\rho}\gamma_s^m$.

Recall the quotient $\mathcal{C}_f'$ of the component of the space of framed curves containing $K$, defined in Section~\ref{generators.sec} above, is a discrete fiber cover of $\mathcal{C}$, and that of the loops $\gamma_{\text{rot}}$,
$\gamma_{\text{fib}}$,$\gamma_{\rho}$, and $\gamma_s$ above, each lifts to a loop in 
 $\pi_1(\mathcal{C}_f',K)$ except $\gamma_s$. (To simplify notation, we denote the lift  $\tilde\gamma\in \pi_1(\mathcal{C}_f',K)$ of each $\gamma\in\pi_1(\mathcal{C},K)$ by $\gamma$, rather than $\tilde\gamma$.)  

Hence for any $\alpha\in\pi_1(\mathcal{C}_f',K)$,  there exists $i\in \mathbb{Z}$ such that

\begin{equation}\label{framed_curves.eq}
\alpha^i=\begin{cases}
 \gamma_{\text{rot}}^a \gamma_{\text{fib}}^b &\text{ if } [K]\neq [f^j]\in\pi_1(M)\\
  \phi\gamma_{\text{rot}}^a  \gamma_\rho   \phi^{-1}& \text{ if } [K]=[f^j]\in\pi_1(M)\\
\end{cases}
\end{equation}

Finally, we derive a statement analogous to Equation~\ref{framed_curves.eq} for $V$-transverse curves.

Given a nowhere-vanishing vector field $V$ on $M$, we can consider both the corresponding $2$-plane bundle on $M$ determined by $V^\perp$, as well as its unit sub-bundle $E_V$.   Every $V$-transverse knot $K$ in $M$ has a lift $\tilde{K}$ to $E_V$ determined by the projection of its tangent vector to $V^\perp$. The space of $V$-transverse immersions $\mathcal{C}_V$ of $S^1\rightarrow M$ is weak homotopy equivalent to the free loop space $\Omega E_V$, since the differential relation of $V$-transversality is ample \cite[Theorem 18.4.1]{eliashberg2002introduction}.  Therefore, to classify loops in $\alpha \in \pi_1(\mathcal{C}_V,K)$, it suffices to classify their lifts $\tilde \alpha \in \pi_1(\Omega E_V,\tilde{K})$.  

As in the immersed case, given a class $\tilde \alpha\in \pi_1(\Omega E_V,\tilde{K})$, we let $t(\tilde \alpha)\in \pi_1(E_V,\tilde{K}(t_0))$ denote the class of the trace of the basepoint $\tilde{K}(t_0)$ during a self-homotopy  of $\tilde{K}$ representing $\tilde \alpha$.  Given $\alpha_1,\alpha_2\in \pi_1(\Omega E_V)$ with $t(\alpha_1)=t(\alpha_2)$, the obstruction for $\alpha_1$ and $\alpha_2$ to be equal in $\pi_1(\Omega E_V)$ is an element of $\pi_2(E_V)$.  Since $\pi_2(M)=1$, and $E_V$ is an $S^1$-bundle over $M$, the long exact sequence of the fibration gives $\pi_2(E_V)=1$ as well.  Hence elements of $\pi_1(\Omega E_V)$ are completely determined by their trace in $\pi_1(E_V)$.  Let $\text{pr}:E_V\rightarrow M$ denote the projection map.  We have an exact sequence $1\rightarrow \pi_1(S^1)\rightarrow \pi_1(E_V)\rightarrow \pi_1(M)\rightarrow 1$, from which it follows that every element $\beta\in \pi_1(E_V)$ is of the form $\gamma g^k$ for $\gamma\in \pi_1(M)$, where $g$ is a generator of $\pi_1(E_V)$ of corresponding to the $S^1$-fiber of $\text{pr}:E_V\rightarrow M$.  An element $\gamma_{\text{kink}}\in \pi_1(\Omega E_V)$ with $t(\gamma_{\text{kink}})=g$ is shown in Figure~\ref{fig:kinkslide}.  In other words, the loop $\gamma_{\text{kink}}$ is achieved by sliding one kink around the image of $K$, after creating a pair of opposite kinks (as shown in Figure \ref{fig:kinks}) if necessary, in order to create a local kink to slide. 

  Next, consider the map $\iota:\mathcal{C}_V\rightarrow \mathcal{C}_f'$,given by composing the inclusion $\mathcal{C}_V\hookrightarrow \mathcal{C}_f$ and the quotient map $\mathcal{C}_f\rightarrow \mathcal{C}_f'$.  We have an induced map $\iota_*:\pi_1(\mathcal{C}_V,K)\rightarrow\pi_1( \mathcal{C},K)$, where $K$ is $V$-transverse.  Since the trace of $\gamma^{kink}$ is the $S^1$-fiber of $E_V$, we have that $\text{ker }\iota_*$ is the infinite cyclic group generated by $\gamma_{\text{kink}}$.  By Proposition~\ref{hV_torus}, we then have an exact sequence of groups 

  \begin{equation}\label{immersed_to_Vtrans_seq.eq}1\rightarrow \langle \gamma_{\text{kink}}\rangle \rightarrow \pi_1(\mathcal{C}_V,K)\rightarrow \pi_1(\mathcal{C},K)/\langle \gamma_s\rangle \xrightarrow{h_V} \langle \mathbb{Z},+\rangle\rightarrow 1.
  \end{equation}

  The result follows from combining Equations \ref{framed_curves.eq} and \ref{immersed_to_Vtrans_seq.eq} and Proposition ~\ref{hV_torus}.\end{proof}

\section{A Finite-Type Invariant of $V$-transverse links}\label{invariant.sec}

\begin{defn} A figure-8 in a manifold $X$ is a map $S^1\vee S^1\rightarrow X.$
    An {\it immersed figure-8} in a 3-manifold $M$ is determined by pair of immersions $A,B:S^1\rightarrow M$ with a distinguished common point $p\in \text{im } A \cap \text{im } B$.  An immersed figure-8 is thus described by a map $A\bullet_p B :S^1\vee S^1 \rightarrow M$, where each component $A$ and $B$ is immersed. We call $A\bullet_p B$ a {\it $V$-transverse figure-8} in $(M,V)$ if each of the components $A$ and $B$ are $V$-transverse.
\end{defn}

We define a function on certain generic homotopies of 2-component $V$-transverse links.  Under certain conditions on the homotopy classes of the link components, this function can be used to define a finite-type invariant of $V$-transverse links analogous to invariants of smooth knots and links defined in \cite{Vassiliev}, \cite{Kirk-Livingston}, and \cite{kalfagianni1998finite}. 

Let $H$ be a generic $V$-transverse homotopy from $L=K_1\cup K_2$ to $L'=K_1'\cup K_2'$. By generic, we mean there exist finitely many times $t_1,\dots t_k$ where $H_t(L)$ is not an embedding, and at each $t_i$, $H_t(L)$ has a single transverse double point $p_i$.  Denote the set of such $p_i$ consisting of double points between {\it distinct} components of $H_t(K_1)$ and $H_t(K_2)$ by $D_{12}$.  Denote the sign of the resolution of such a double point by $\epsilon(t_i)$, where $\epsilon(t_i)$ is positive (resp. negative) if we pass from a negative (resp. positive) crossing to a positive (resp. negative) crossing, where positive and negative crossings are modeled locally as in Figure \ref{fig:crossing}.

\begin{defn}\label{invariant.def}Let $H$ be a generic $V$-transverse homotopy from  $L=K_1\cup K_2$ to $L'=K_1'\cup K_2'$, with double point set $D_{12}$ as defined above.  We define $\nu(H)$ as follows:
    $$\nu(H)=\sum_{p_i \in D_{12}} \epsilon(t_i) [H_t(K_1)\bullet_{p_i} H_t(K_2)] . $$
\end{defn}

The equivalence relation on $V$-transverse figure-8's is as follows.  Given $A\bullet_p B$ and $C\bullet_p D$, we have  $[A\bullet_p B]=[C\bullet_p D]$ if there is a homotopy through $V$-transverse figure 8's carrying $A\bullet_p B$ to $C\bullet_p D$. In the following paragraphs, we formulate the existance of such a homotopy in terms of a computable algebraic criterion.

Recall from the proof of Lemma~\ref{lem:Vtranscurveclassification} that $E_V$ denotes the unit sub-bundle of $V^\perp$, and hence determines an $S^1$-bundle on $M.$  Given a $V$-transverse immersed curve $C$, we denote its lift to $E_V$, determined by projecting its tangent vector to $V^\perp$ at each point, by $\tilde C.$

If $A$ and $B$ intersect at $p$, their lifts $\tilde{A}$ and $\tilde{B}$ generally won't intersect in $E_V$, but  we can construct a corresponding (topological) figure-8 in $E_V$ as follows.  Let $q_A$ denote the lift to $E_V$ of the tangent vector $A'$ of $A$ at $p$ , and similarly let $q_B$ denote the lift of $B'$ at $p$ to $E_V$.  Let $\delta_{AB}$ be a path from $q_A$ to $q_B$ that follows the orientation of the $S^1$ fiber of $E_V$.  Then form a figure-8 in $E_V$, $\tilde{A}\bullet_{q_A}\delta_{AB}\tilde{B}\delta_{AB}^{-1}$.  Note that the choice of path from $q_A$ to $q_B$ in the fiber does not matter, because the class of the fiber is in the center of $\pi_1(E_V)$.

\begin{prop} \label{term_cancellation.prop} Let $A \bullet_p B$ and $C\bullet_p D$ be $V$-transverse figure 8's in $(M,V)$ such that the tangent vectors of $A$ and $C$ agree at $p$, and such that the tangent vectors of $B$ and $D$ agree at $p$.
Further, suppose $A \bullet_p B$ and $C\bullet_p D$ are homotopic through $V$-transverse figure-8's. Then there exists $\gamma\in \pi_1(E_V,q_A)$ such that $\gamma \tilde A \gamma ^{-1}=\tilde C $ and  $\gamma  \delta_{AB}\tilde{B}\delta_{AB}^{-1}\gamma^{-1} =\delta_{AB}\tilde{D}\delta_{AB}^{-1} $.
\end{prop}

The following proposition shows how to use $\nu$ to obstruct the existence of a $V$-transverse link-homotopy between two $V$-transverse links $L$ and $L'.$  Note that an operation similar to $\nu$ can be defined for individual components of a link, as a sum of double points over self-intersections of components during a generic homotopy.  Such operations can be used to obstruct the existence of a component-wise $V$-transverse isotopy.  Both types of operations can also obstruct the existence of a $V$-transverse link isotopy.  Proposition~\ref{nu_invariant.prop} describes a necessary condition for using $\nu$ as such an obstruction.

\begin{prop}\label{nu_invariant.prop} Suppose that $\nu(H_1 \sqcup H_2)=0$ for all $V$-transverse self-homotopies $H_1\sqcup H_2$ of a link $L=K_1\sqcup K_2$ with components in  the $V$-transverse homotopy classes $\mathcal{C}_{V,1}$ and $\mathcal{C}_{V,2}$.  Then for any two $V$-transverse links $L=K_1\sqcup K_2$ and $L'=K_1' \sqcup K_2'$ that (1) have components in $\mathcal{C}_{V,1}$ and $\mathcal{C}_{V,2}$ and (2) are link-homotopic as $V$-transverse links,  $\nu(H_1' \sqcup H_2')=0$ for any $V$-transverse homotopy $H_1' \sqcup H_2'$ from $L$ to $L'.$ 
\end{prop}

\begin{proof} Let $L$, $L'$ be $V$-transverse links with components in $\mathcal{C}_{V,1}$ and $\mathcal{C}_{V,2}$ be link-homotopic as $V$-transverse links. Suppose  $\nu(H_1 \sqcup H_2)=0$ for all $V$-transverse self-homotopies $H_1\sqcup H_2$ of $L$.  Let $H_1' \sqcup H_2'$ be a $V$-transverse homotopy from $L$ to $L'$.  Let $H_1'' \sqcup H_2''$ be a $V$-transverse link-homotopy from $L$ to $L'$.  It follows directly from the definition of $\nu$ that $\nu(H_1'' \sqcup H_2'')=0.$  Concatenating $H_1' \sqcup H_2'$ with the reverse of $H_1'' \sqcup H_2''$ gives the desired result.\end{proof}

In Section~\ref{examples.sec} we characterize certain $V$-transverse homotopy classes in an $S^1$-bundle over an oriented surface for which the hypotheses of Proposition~\ref{nu_invariant.prop} hold, meaning $\nu$ can be used as an invariant.

\section{Examples}\label{examples.sec}

Throughout this section, $M$ is an $S^1$-bundle over an oriented surface $F$ of genus $\geq 2$.  Our goal in this section is to construct examples of $V$-transverse links in $M$ which are isotopic as framed links, homotopic as $V$-transverse immersed curves, and not isotopic as $V$-transverse knots.  In other words, these framed link types are non-simple in the $V$-transverse category.  Furthermore, the various link types represent various types of non-simplicity phenomena, as summarized in table ~\ref{summary.tab}.

In each example, both components of the link are homotopic to the $S^1$-fiber of $M$.  We begin by showing that the operation $\nu$ defined in Section ~\ref{invariant.sec} can be used as an invariant in this case. In particular, by Proposition~\ref{nu_invariant.prop}, Proposition~\ref{well_defined_both_fiber.prop} implies that the value of $\nu$ is independent of the choice of $V$-transverse homotopy between pairs of links with components homotopic to the fiber, and hence can be used to obstruct the existence of a $V$-transverse link homotopy between such links. We note that the same proof works when the components are homotopic to nontrivial powers of the fiber.

\begin{prop} \label{well_defined_both_fiber.prop}Let $M$ be an $S^1$-bundle over an oriented surface $F$, and $V$ be a nowhere zero vector field on $M$.  Let $K_1\sqcup K_2$ be a $V$-transverse link.  If $K_1$ and $K_2$ are both homotopic to the $S^1$-fiber $f$ of $M$, then $\nu(H_1\sqcup H_2)=0$ for all $V$-transverse self-homotopies of $K_1\sqcup K_2$.\end{prop}

\begin{proof}

  If $K_1$ and $K_2$ are both homotopic to the $S^1$-fiber $f$ of $M$ then by Lemma~\ref{lem:Vtranscurveclassification}, any $V$-transverse self-homotopy of $K_1\sqcup K_2$ is of the form $H_1\sqcup H_2$, where $H_1=\phi_1\gamma_{\rho_1}\gamma_{\text{rot}}^a\gamma_{\text{kink}}^b\phi_1^{-1}$ and $H_2=\phi_2\gamma_{\rho_2}\gamma_{\text{rot}}^c\gamma_{\text{kink}}^d\phi_2^{-1}$ are $V$-transverse self-homotopies of $K_1$ and $K_2$ respectively, with $\rho_i\in \pi_1(F)$, for $i=1,2$. (Note that the evaluation of $e_{V^\perp}$ on the torus corresponding to $\gamma_{\rho_i}$ is $0$).  We may also assume that $H_1$ and $H_2$ are performed sequentially, with $H_1$ occuring before $H_2$.  
  First note no double points of $K_1$ and $K_2$ occur during $\gamma_{\text{kink}}$ or  $\gamma_{\text{rot}}$.  Also note that $H_1$ may be assumed to occur before $H_2$. Hence any intersections of $K_1$ and $K_2$ during $H_1\sqcup H_2$ can be assumed to occur during $\phi_1^{\pm 1}$ and $\phi_2^{\pm 1}$ when $H_1$ and $H_2$ are performed sequentially.  Therefore all terms of $\nu(H_1\sqcup H_2)$ occur twice with opposite sign, and so $\nu(H_1\sqcup H_2)=0$ as desired.

Following Arnold, it remains to check that $\nu$ vanishes on small loops around the codimension 2 stratum in $\mathcal{C}_V$ as in \cite{arnold1994topological}.  The only codimension 2 strata that arise are formed by $V$-transverse curves with exactly two double points. When going around the stratum, we see four singular curves with one double point separated into pairs with opposite signs, so $\nu$ indeed vanishes when going around this stratum.  For a similar argument, see \cite[Theorem 4.3.1 and Remark 4.2.2]{tchernov2002arnold}.
\end{proof}

\begin{ex}\label{parallel_fiber.ex}   {\bf Non-simple links with parallel components}

\begin{figure}
\labellist
\small\hair 2 pt
\pinlabel \textcolor{blue}{$K_1$} at 200 400
\pinlabel \textcolor{red}{$K_2$} at 350 400
\pinlabel $\pi(d)$ at 480 100
\endlabellist
\centering
    \includegraphics[width=0.5\linewidth]{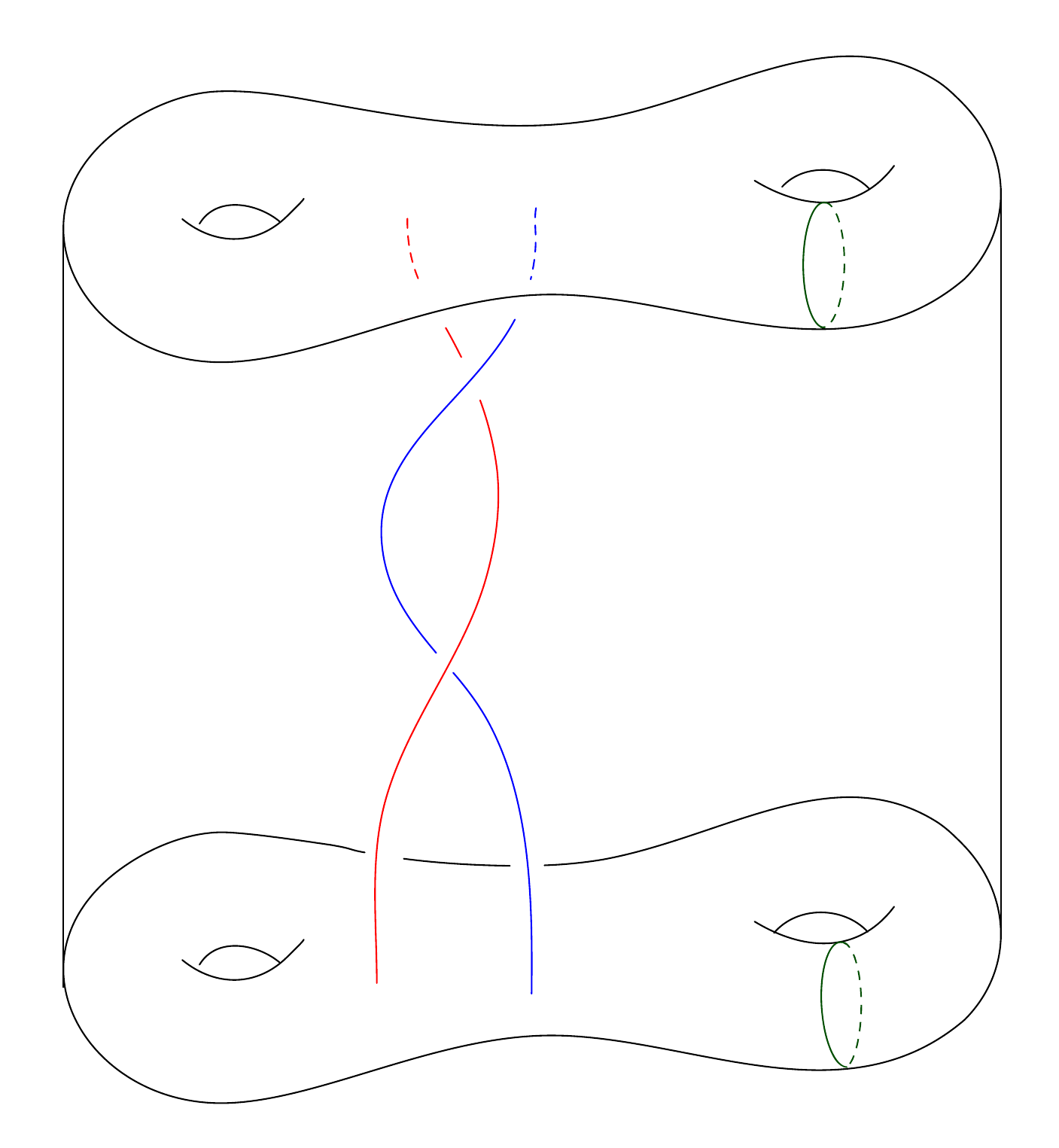}
    \caption{A non-simple class represented by a pair of parallel components in $M=F\times S^1$ where $F$ is a surface of genus 2.  A similar construction works for $S^1$ bundles over oriented surfaces.  See Section~\ref{parallel_fiber.ex}.  The Euler class $e_{V_k^\perp}\in H^2(F\times S^1)$ is dual to $2k\cdot[d]\in H_1(F\times S^1)$.}
    \label{fig:parallel_fiber}
\end{figure}

Let $M$ be an $S^1$-bundle over an oriented surface $F$ of genus $\geq  2$ with projection map $\pi:M\rightarrow F$.  Let $K_1$ and $K_2$ be once-linked curves in an torus neighborhood of an $S^1$-fiber $f$ of $M$ as in Figure ~\ref{fig:parallel_fiber} for the case $F\times S^1$. We assume that $K_1=f$. Let $d$ be a curve in $M$ which projects to an essential simple closed curve $\pi(d)$ on $F$ disjoint from the projections of $K_1$ and $K_2$ to $F$.  Let $V_k$ be a nowhere-zero vector field on $M$ such that the Euler class of $V_k^\perp$ is Poincar\'e dual to $2k\cdot[d]\in H_1(M)$, with $k\;\mathbb{Z}\setminus \{0\}$; the curve $\pi(d)$ (for some choice of $d$) is also shown in Figure ~\ref{fig:parallel_fiber}.  Perturb $K_1$ and $K_2$ so that they are transverse to $V_k$.  Recall that the $V_k$-transverse link obtained by adding $i_1$ signed kink pairs to $K_1$ and $i_2$ signed kink pairs to $K_2$ is denoted $K_1^{i_1}\sqcup K_2^{i_2}$.

\begin{prop}\label{prop:parallel_fiber} Let $(M,V_k)$ and $K_1\sqcup  K_2$ be as defined above.
    For $i_1,i_2\in k \;\mathbb{Z}$, the links $K_1\sqcup K_2$ and $K_1^{i_1}\sqcup K_2^{i_2}$ above are framed isotopic and homotopic as $V_k$-transverse immersed curves.  They are isotopic as $V_k$-transverse links if and only if $i_1=i_2$. Moreover, when $i_1\neq i_2$ the links are component-wise $V_k$-transverse isotopic, but the links are not link-homotopic as $V_k$-transverse links.
    
\end{prop}

\begin{rem}In the theorem above, $K_1\sqcup K_2$ and $K_1^{i_1}\sqcup K_2^{i_2}$ are not homotopic as $V_k$-transverse multi-curves when $i_1,i_2\notin k\;\mathbb{Z}$, so the theorem covers all cases of interest.
\end{rem}

\begin{proof}
{\bf Case 1} $\mathbf{i_1\neq i_2}$.  We assume $K_1$ is the $S^1$-fiber $f$ of $M$.  Let $\phi$ be a $V$-transverse homotopy of $K_2$ to a second copy of the vertical fiber $f'$, during which $K_2$ passes through $K_1$ once positively.  Then apply the framed self-homotopy $\gamma_{\rho_1}$ (as defined in Section~\ref{generators.sec}) to $K_1$, where $\rho_1\in \pi_1(F,p)$ is a curve on $F$ intersecting a reprsentative of the class $k\cdot [\pi(d)]\in H^1(F;\mathbb{Z})$  $i_1$ times algebraically.  Next we approximate this framed self-homotopy of $K_1$ by a $V$-transverse homotopy $(\gamma_{\rho_1})_V$ from $K_1$ to $K_1^{i_1}$ as in Proposition ~\ref{hV_torus}.  We do the same for $K_2$, by choosing a curve $\rho_2$ intersecting a representative of $k\cdot[\pi(d)]\in H_1(F)$ $i_2$ times algebraically.  Finally we apply $\phi^{-1}$ to re-clasp $K_1^{i_1}$ and $K_2^{i_2}$.

Next we use the invariant $\nu$ to obstruct the $V_k$-transverse homotopy $ H=(\gamma_{\rho_1})_V\sqcup \phi (\gamma_{\rho_2})_V\phi^{-1}$ from being homotopic to a $V_k$-transverse link-homotopy (and hence from being homotopic to a $V_k$-transverse link isotopy).

Applying the formula in Definition \ref{invariant.def} as well as Proposition ~\ref{hV_torus}, we have 
$$\nu(H)=[K_1\bullet_q K_2]-[K_1^{i_2}\bullet_q K_2^{i_2}].$$

We will show the two terms of $\nu(H)$ do not cancel, and hence $\nu(H_1\sqcup H_2)\neq 0$, using Proposition ~\ref{term_cancellation.prop}.  Let $\tilde{K}_i$ denote the lift of $K_i$ to the $S^1$-bundle $E_V$ over $M$ defined in the proof of Lemma~\ref{lem:Vtranscurveclassification}. Let $F$ denote the $S^1$-fiber of this bundle. The terms $K_1\bullet_q K_2$ and $K_1^{i_2}\bullet_q K_2^{i_2}$ lift to pairs $(\tilde{K}_1,\tilde{K}_2)$ and $(F^{2i_1}\tilde{K}_1,F^{2i_2}\tilde{K}_2)$.  By Proposition~\ref{term_cancellation.prop}, if these terms cancel, there is an element $\gamma\in \pi_1(E_V)$ such that $\gamma \tilde{K}_1 \gamma^{-1}=F^{2i_1}\tilde{K}_1$ and $\gamma \tilde{K}_2 \gamma^{-1}=F^{2i_2}\tilde{K}_2$. Furthermore, since both ${K}_i$ are homotopic to the $S^1$-fiber $f$ of $M$ we can write $\tilde{f}=\tilde{K}_1=\tilde{K_2},$ giving the pair of equations

$$\gamma\tilde{f}\gamma^{-1}=F^{2i_1}\tilde{f},$$
and 
$$\gamma\tilde{f}\gamma^{-1}=F^{2i_2}\tilde{f}.$$

Hence $F^{2i_1}=F^{2i_2}$.  By the long exact sequence of the fibration $S^1\hookrightarrow E_V\rightarrow M$, as $\pi_2(M)=1$, $F$ has infinite order in $\pi_1(E_V)$.  Thus we have $i_1=i_2$, a contradiction.  Hence $\nu(H)\neq 0$.

Finally, by Propositions~\ref{nu_invariant.prop} and ~\ref{well_defined_both_fiber.prop} $K_1\sqcup K_2$ and $K_1^{i_1}\sqcup K_2^{i_2}$ are not $V_k$-transverse link-homotopic when $i_1\neq i_2.$

{\bf Case 2:} $\mathbf{i_1=i_2}$. When $i_1=i_2$ we can find a $V_k$-transverse isotopy from  $K_1\sqcup K_2$ to $K_1^{i_1}\sqcup K_2^{i_2}$ by simultaneously applying the same $\gamma_{\rho}$ to $K_1\sqcup K_2$ without unclasping $K_1$ and $K_2$, where $\rho$ is any curve on $F$ intersecting a representative of $k\cdot[\pi(d)]\in H_1(F)$ $i_1=i_2$ times algebraically.\end{proof}

\end{ex}

\begin{ex}\label{clasped_fiber.ex}   {\bf Non-simple links with clasped and parallel components}

\begin{figure}
\labellist
\small\hair 2 pt
\pinlabel \textcolor{blue}{$K_1$} at 100 450
\pinlabel \textcolor{red}{$K_2$} at 280 390
\pinlabel $\pi(d)$ at 490 100
\endlabellist
    \centering
    \includegraphics[scale=0.4]{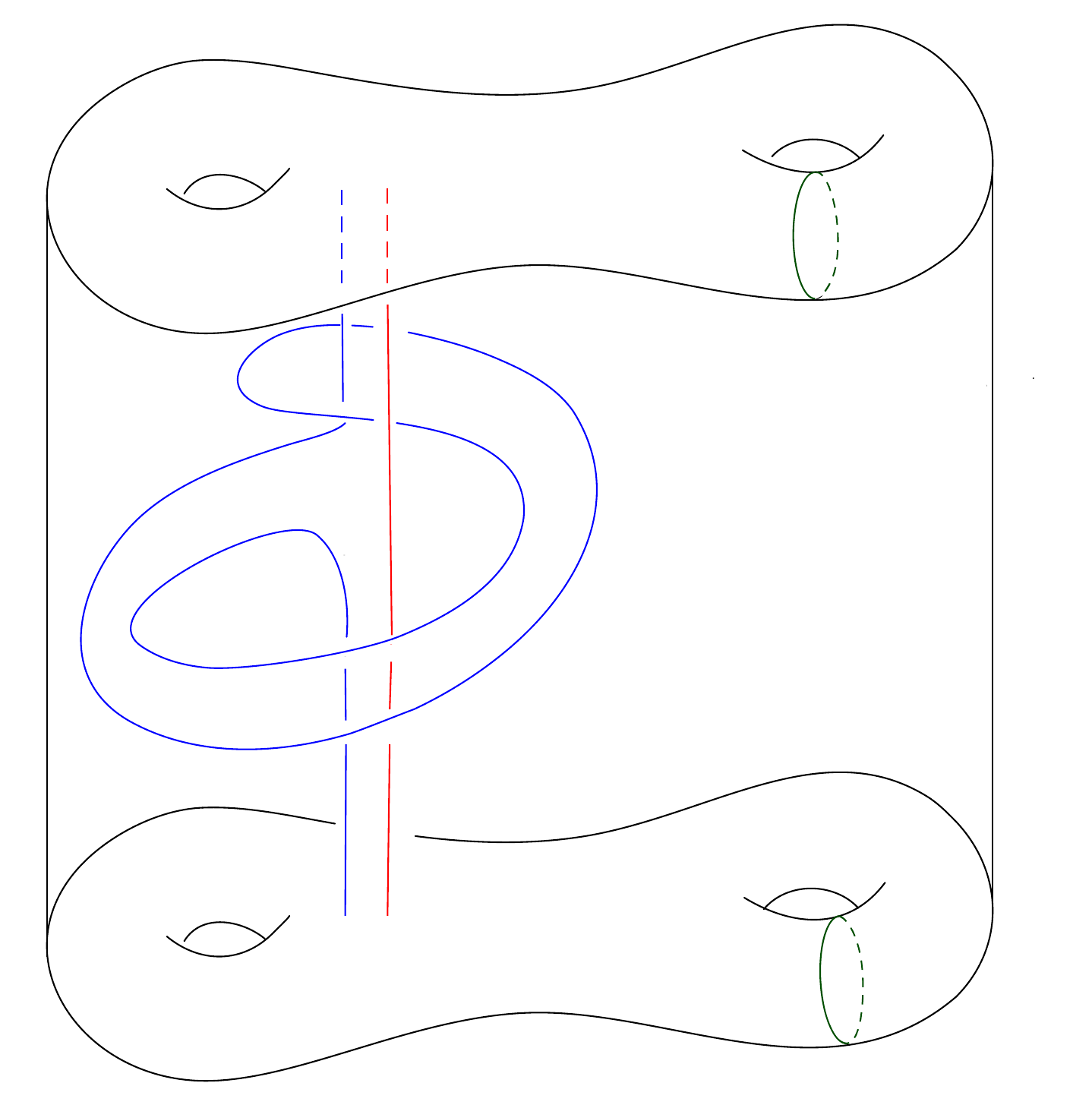}
    \caption{Schematic of a link $K_1\sqcup K_2$ in $F\times S^1$ such that $K_1\sqcup K_2$ and $K_1^{k_1}\sqcup K_2^{k_2}$ are framed isotopic, $V_k$-transversely homotopic, $V_k$-transversely link-homotopic, and not $V_k$-transversely isotopic, when $k_1\neq 0$.  The Euler class $e_{V_k^\perp}\in H^2(F\times S^1)$ is dual to $2k\cdot[d]\in H_1(F\times S^1)$.}
    \label{fig:clasped_fiber}
\end{figure}

Let $M$ be an $S^1$-bundle over an oriented surface $F$ of genus $\geq 2$, with projection map $\pi:M\rightarrow F$. Let $K_1$ and $K_2$ be curves homotopic to the $S^1$-fiber $f$ of $M$, where $K_1$ is clasped with itself as shown in Figure \ref{fig:clasped_fiber} in the case $F\times S^1$. Let $d$ be a curve in $M$ which projects to an essential simple closed curve $\pi(d)$ on $F$ disjoint from the projections of $K_1$ and $K_2$ to $F$. The vector field $V_k$ is again chosen so that $e_{V_k^\perp}=2k\cdot[d]\in H_1(M,\mathbb{Z})$.  Also as before, we perturb $K_1$ and $K_2$ to be $V_k$-transverse. We may assume that $K_2=f$.

\begin{prop}\label{clasped.prop} Let $(M,V_k)$ and $K_1\sqcup K_2$ be as defined above.  Let $i_1,i_2\in k\; \mathbb{Z}$, with $i_1\neq 0$.  The links $K_1\sqcup K_2$ and $K_1^{i_1}\sqcup K_2^{i_2}$ are framed isotopic, homotopic as $V_k$-transverse immersed multi-curves and link-homotopic as $V_k$-transverse multi-curves. But they are not component-wise isotopic as $V_k$-transverse links, and hence not isotopic as $V_k$-transverse links.

\end{prop}

\begin{rem}
     As for the case when $i_1=0$, the links $K_1\sqcup K_2$ and $K_1\sqcup K_2^{i_2}$ are framed isotopic and homotopic as $V_k$-transverse multi-curves. We suspect that $K_1\sqcup K_2$ and $K_1\sqcup K_2^{i_2}$ are not isotopic as $V_k$-transverse links, but $\nu$ fails to obstruct such an isotopy.\\
     When $i_1,i_2\notin k\; \mathbb{Z}$ the above links are not homotopic as $V_k$-transverse multi-curves, so the above theorem covers all cases of interest aside from the case where $i_1=0$.
\end{rem}
\begin{proof}

First we construct a $V_k$-transverse homotopy from $K_1\sqcup K_2$ to $K_1^{i_1}\sqcup K_2^{i_2}.$

Let $\sigma$ be a $V_k$-transverse homotopy that unclasps $K_1$. (Note that, after applying $\sigma$, we see that $K_1$ and $K_2$ are both isotopic to $f$.) Then apply the framed self-homotopy $\gamma_{\rho_1}$ to $K_1,$ where $\rho_1\in \pi_1(F,p)$ is a curve on $F$ intersecting a representative of the class $k\cdot[\pi(d)]\in H_1(F)$ $i_1$ times algebraically. Next we approximate  this self-homotopy of $K_1$ by a $V_k$-transverse homotopy $(\gamma_{\rho_1})_V$ from $K_1$ to $K_1^{i_1}$ using Proposition~\ref{hV_torus}. We do the same for $K_2$ by choosing a curve $\rho_2$ intersecting a representative of $k\cdot[\pi(d)]$ $i_2$ times. Finally, we apply $\sigma^{-1}$ to re-clasp $K_1$.

 Note that, we have $H=\sigma(\gamma_{\rho_1})_V\sigma^{-1}\sqcup (\gamma_{\rho_2})_V$. During this $V_k$-transverse homotopy no double points occur between the two components, this is also a $V_k$-transverse link-homotopy.  
 
 According to \cite[Thoeorem 10.5]{CC}, there exists no $V_k$-transverse isotopy between $K_1$ and $K_1^{i_1}$ for $i_1\neq 0.$ The links $K_1\sqcup K_2$ and $K_1^{i_1}\sqcup K_2$ are not component-wise $V_k$-transversely isotopic, and hence are not $V_k$-transversely isotopic as links.

 These are examples of weakly simple $V$-transverse links.
    
\end{proof}

\end{ex}

\begin{ex}\label{clasped_twisted_fiber.ex}{\bf Non-simple links with clasped and braided components}

The purpose of our final example is to show that it is possible for all of the above non-simplicity phenomena to occur within a {\it single} framed isotopy class of links.

Let $M$ be an $S^1$-bundle over an oriented surface $F$ of genus $g\geq 2$ as before.  Let $K_1\sqcup K_2$ be as shown in Figure~\ref{fig:clasped_braided}. Namely, $K_1$ is a curve parallel to the fiber with the exception of one self-clasp, as in Example~\ref{clasped_fiber.ex}, as well as linked with $K_2$ exactly once in a torus neighborhood of an $S^1$-fiber $f$ of $M$, as in Example ~\ref{parallel_fiber.ex}.  The curve $d$ and vector field $V_k$ are defined as in Example~  \ref{clasped_fiber.ex}.
 We perturb $K_1\sqcup K_2$ so that they are $V_k$-transverse, and assume that $K_2=f$.  

\begin{figure}
\labellist
\small\hair 2 pt
\pinlabel \textcolor{blue}{$K_1$} at 120 480
\pinlabel \textcolor{red}{$K_2$} at 340 460
\pinlabel $\pi(d)$ at 500 150
\endlabellist
    \centering
    \includegraphics[scale=1.3,width=4in]{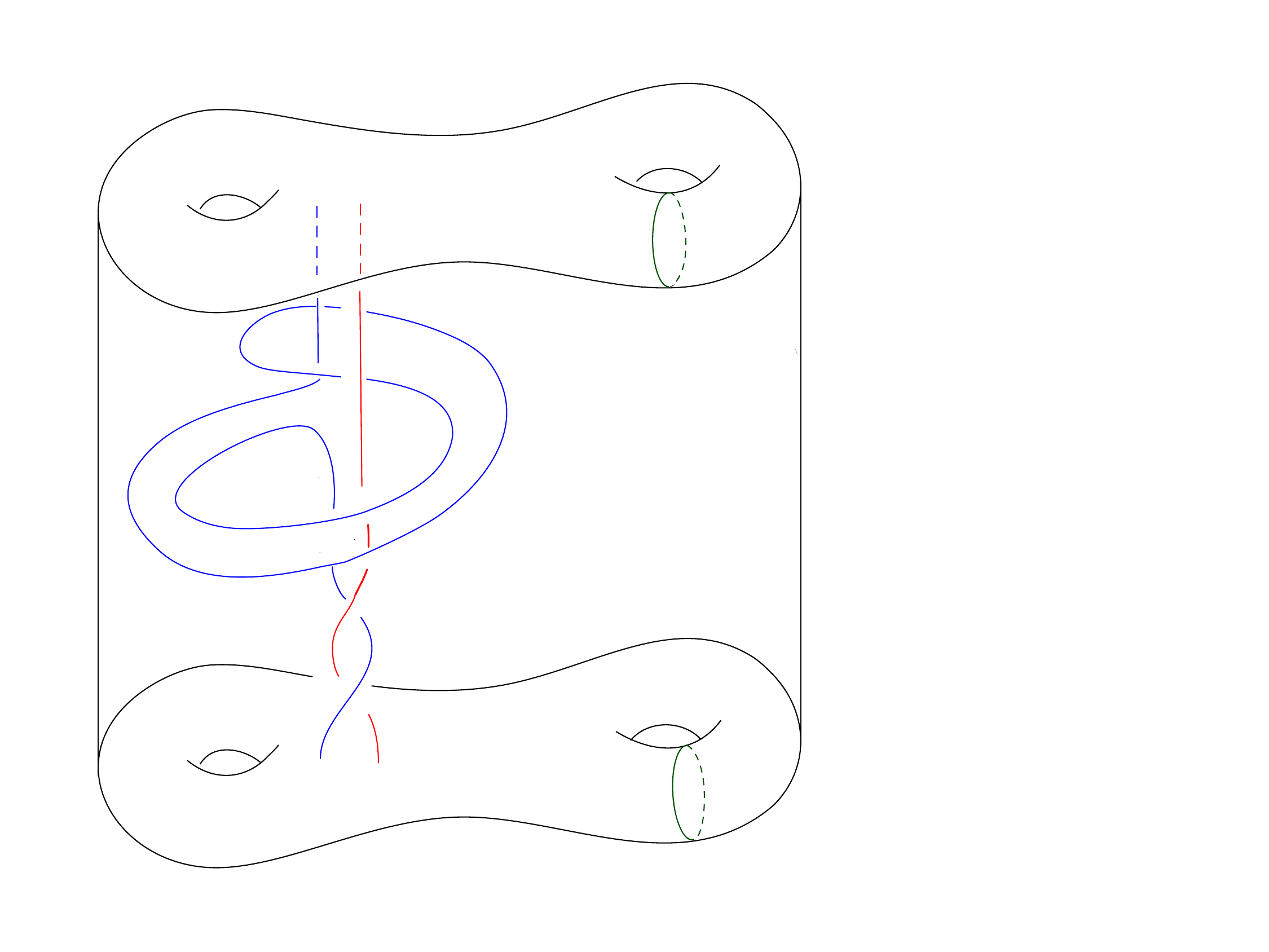}
    \caption{Schematic of a link $K_1\sqcup K_2$ in $F\times S^1$ such that $K_1\sqcup K_2$ and $K_1^{k_1}\sqcup K_2^{k_2}$ are framed isotopic, $V_k$-transversely homotopic, not $V_k$-transversely link-homotopic, and not $V_k$-transversely isotopic, when $k_1\neq 0$ and $k_1\neq k_2$.  The Euler class $e_{V_k^\perp}\in H^2(F\times S^1)$ is dual to $2k\cdot[d]\in H_1(F\times S^1)$. }
    
    \label{fig:clasped_braided}
\end{figure}

\begin{prop}
    
\label{ex4.prop} Let $i_1,i_2\in k\;\mathbb{Z}$.  The links $K_1\sqcup K_2$ and $K_1^{i_1}\sqcup K_2^{i_2}$ as  defined above are framed isotopic and homotopic as $V_k$-transverse immersed curves. These links are
\begin{enumerate}
    \item $V_k$-transverse link-homotopic and not component-wise $V_k$-transverse isotopic (and hence not $V_k$-transverse isotopic) if $i_1=i_2\neq 0$
    \item neither $V_k$-transverse link-homotopic nor component-wise $V_k$-transverse (and hence not $V_k$-transverse isotopic) if $i_1\neq i_2$ and $i_1\neq 0$
    \item component-wise $V_k$-transverse isotopic but not $V_k$-transverse link-homotopic (and hence not $V_k$-transverse isotopic) if $i_1=0$ and $i_2\neq 0$
\end{enumerate}

\end{prop}

\begin{proof}
The proofs of each item above are similar to analogous arguments in the proofs of Propositions~\ref{prop:parallel_fiber} and ~\ref{clasped.prop}.
\end{proof}

\end{ex}

\subsection*{Legendrian Case}

We will construct Legendrian links satisfying non-simplicity phenomena in a family of contact manifolds $(M,\xi_k)$, analogous to the phenomena in Examples ~\ref{parallel_fiber.ex}, ~\ref{clasped_fiber.ex}, and ~\ref{clasped_twisted_fiber.ex}.

We begin stating a collection of results relating homotopy and isotopy in the $V$-transverse and Legendrian categories. A version of $(4)$ is formulated for $V$-transverse knots in \cite{CC}, is formulated in the formal Legendrian category in \cite{cieliebak2012stein}, as well as proven in other cases in increasing generality in  \cite{dymara2001legendrian,eliashberg2009topologically,ding2004surgery}. 

\begin{thm}\label{Vtrans_to_Leg_promotion.thm}
Let $L_1$ and $L_2$ be Legendrian links in $(M,\xi)$, with $M$ closed and orientable, such that $\xi$ is co-oriented, with $V$ any choice of co-orienting vector field for $\xi.$ 
\begin{enumerate}
\item  $L_1$ and $L_2$ are homotopic as $V$-transverse multi-curves if and only if $L_1$ and $L_2$ are homotopic as Legendrian multi-curves.
\item $L_1$ and $L_2$ are link-homotopic as $V$-transverse links if and only if $L_1$ and $L_2$ are link-homotopic as Legendrian links.
\item If there is an overtwisted disk $D$ in $(M,\xi)$ such that $L_1,L_2\subseteq M\setminus D$, then $L_1$ and $L_2$ are component-wise isotopic as $V$-transverse links if and only if $L_1$ and $L_2$ are component-wise isotopic as Legendrian links. \cite[Theorem 7.19]{cieliebak2012stein}
\item If there is an overtwisted disk $D$ in $(M,\xi)$ such that $L_1,L_2\subseteq M\setminus D$, then $L_1$ and $L_2$ are  isotopic as $V$-transverse links if and only if $L_1$ and $L_2$ are isotopic as Legendrian links. \cite[Theorem 7.19]{cieliebak2012stein}
\end{enumerate}
\end{thm}
\begin{proof}
$(1)$:  Let $L_1=A_1\sqcup A_2$ $L_2=B_1\sqcup B_2$ be Legendrian links in $(M,\xi)$ which are homotopic as $V$-transverse multi-curves.  After sufficiently many stabilizations of the components of $L_1$ and $L_2$, any homotopy from $L_1$ and $L_2$ carrying $A_1$ to $B_1$ and $A_2$ to $B_2$ can be $C^0$-approximated by a Legendrian homotopy \cite{gromov1986partial,eliashberg1992contact}. Recall from Section~\ref{stab.sec} and Figure ~\ref{fig:stab} that $L_{i,j}$ denotes a Legendrian link component $L$ with $i$ positive and $j$ negative stabilizations. In particular, for fixed $i=1,2$ (we repeat the argument once for each component), for any  $n_1,n_2\in \mathbb{Z}^+$, there exist $n_3, n_4\in \mathbb{Z}^+$ such that $(A_i)_{n_1,n_2}$ is Legendrian homotopic to $(B_i)_{n_3,n_4}$, and this Legendrian homotopy is $C^0$-close to the given $V$-transverse homotopy from $A_i$ to $B_i$; see also \cite{fuchs1995invariants,goryunov2001plane,hill1997vassiliev,tchernov2003isomorphism}. Let $(A_i)_t$ the homotopy from $A_i$ to $B_i$ at time $t$, and similarly let $((A_i)_{n_1,n_2})_t$ denote the homotopy from $(A_i)_{n_1,n_2}$ to  $(B_i)_{n_3,n_4}$ at time $t$.  By comparing the framing of $(A_i)_t$ to the framing of $((A_i)_{n_1,n_2})_t$ throughout the homotopy, one can show that $n_1+n_2=n_3+n_4$; by comparing the homotopy classes of the lifts of $(A_i)_t$ and $((A_i)_{n_1,n_2})_t$ to $E_V$, one can show that $n_1-n_2=n_3-n_4$. Details of this argument are given in \cite[Section 5.3]{tchernov2003isomorphism}.  
It follows that $n_1=n_3$ and $n_2=n_4.$ Taking $n=\text{max}(n_1,n_2)$, we can perform extra stabilizations on the initial and final components to get a Legendrian homotopy from $(A_i)_{n,n}$ to $(B_i)_{n,n}$.   As shown in \cite[Figure 6]{tchernov2003isomorphism}, any Legendrian curve $L$ is Legendrian homotopic to $L_{n,n}$ for all $n$; this homotopy is performed locally on an arc of $L$.  Hence for each fixed $i=1,2$ we get a Legendrian homotopy from $A_i$ to $B_i$, so $L_1$ and $L_2$ are Legendrian homotopic.  Conversely, if $L_1$ and $L_2$ are homotopic as Legendrian multi-curves, they are clearly $V$-transverse homotopic, as $V$ co-orients $\xi.$

$(2)$:  Let $L_1=A_1\sqcup A_2$ and  $L_2=B_1\sqcup B_2$ be Legendrian links in $(M,\xi)$ which are link-homotopic as $V$-transverse multi-curves. Apply the construction in $(1)$ to get a Legendrian homotopy from $(A_i)_{n,n}$ to $(B_i)_{n,n}$ for $i=1,2$.  Since this Legendrian homotopy is constructed to be $C^0$-close to the given $V$-transverse link-homotopy, it is a Legendrian link-homotopy.  Finally, no double points between components are created during the local Legendrian homotopies from $A_i$ to $(A_i)_{n,n}$ for $i=1,2$.  Therefore, the procedure in $(1)$ yields a Legendrian link-homotopy from $L_1$ to $L_2$.  Conversely, if $L_1$ and $L_2$ are link-homotopic as Legendrian multi-curves, they are clearly $V$-transverse link-homotopic, as $V$ co-orients $\xi.$

In the formal Legendrian category, $(3)$ and $(4)$ are proven in \cite[Theorem 7.19]{cieliebak2012stein}.  In \cite{CC}, an analogous statement is proven for $V$-transverse knots by applying \cite[Theorem 7.19]{cieliebak2012stein}, but the argument for $V$-transverse links is analogous.
\end{proof}

We now return to the specific case where $M$ is an oriented $S^1$-bundle over an oriented surface $F$ of genus $\geq 2$.  Consider one of the links $K_1\sqcup K_2$ in Examples ~\ref{parallel_fiber.ex},~\ref{clasped_fiber.ex}, or ~\ref{clasped_twisted_fiber.ex}.  As in those examples, let $d$ be a curve in $M$ such that its projection $\pi(d)$ is an essential simple closed curve on $F$ disjoint from the projections of $K_1$ and $K_2$ to $F$.
By work of Lutz and Martinet, we first choose a contact structure $\xi_k$ on $M$ such that the Euler class $e_{\xi_k}$ is dual to $2k\cdot[d]\in H_1(M,\mathbb{Z})$, where $d$ is a curve in $M$, and $\pi(d)$ is as shown in Figures ~\ref{fig:parallel_fiber}, ~\ref{fig:clasped_fiber}, and ~\ref{fig:clasped_braided} \cite{lutz1971quelques,martinet2006formes}; see also \cite{ginzburg1988closed,geiges2008introduction}.  Moreover, given any Legendrian link $L_1\sqcup L_2$ in $M$, we can modify $\xi_k$ by performing a full Lutz along a transverse unknot in the complement of $L_1\sqcup L_2$, in such a way that the new contact structure $\xi_k'$ is homotopic to $\xi_k$ as a 2-plane field, and such that $L_1\sqcup L_2$ is a {\it loose} link in $(M,\xi_k')$, i.e., there is an overtwisted disk in its complement \cite[Proposition 4.5.4]{geiges2008introduction}.

\begin{cor}
  Let $M$ be an oriented $S^1$-bundle over an oriented surface $F$ of genus $\geq 2$, with contact structure $\xi_k$ as defined above.  There exist Legendrian links $L$ and $L'$ in $(M, \xi_k)$ such that  $L$ and $L'$  are homotopic as Legendrian multi-cuves, and isotopic as framed links, but

\begin{enumerate}
    \item $L$ and $L'$ are not Legendrian link-homotopic, and thus not Legendrian isotopic as links. These links are certain Legendrian representatives of the isotopy classes $K_1\sqcup K_2$ and $K_1^{i_1}\sqcup K_2^{i_2} $ in Examples~\ref{parallel_fiber.ex} or ~\ref{clasped_twisted_fiber.ex}, with $i_1\neq i_2$. \\
    Moreover, when the above Legendrian links $L$ and $L'$ are representatives of the isotopy classes in Example ~\ref{parallel_fiber.ex} for any $i_1\neq i_2$, or in Example $\ref{clasped_twisted_fiber.ex}$ if $i_1=0$ and $i_2\neq 0$, after modifying $\xi_k$ by a Lutz twist,  $L$ and $L'$ are component-wise Legendrian isotopic.

    \item   $L$ and $L'$  are not component-wise Legendrian isotopic, and thus not Legendrian isotopic as links. Thus they are not Legendrian isotopic as links. These links are certain Legendrian representatives of the isotopy classes $K_1\sqcup K_2$ and $K_1^{i_1}\sqcup K_2^{i_2} $ in Example~\ref{clasped_fiber.ex}, when $i_1\neq 0$, or Example ~\ref{clasped_twisted_fiber.ex}, when $i_1\neq 0$.  \\
   Moreover, $L$ and $L'$  are Legendrian link-homotopic in Example~\ref{clasped_fiber.ex} when $i_1\neq 0$, and in Example ~\ref{clasped_twisted_fiber.ex}, when $i_1=i_2\neq 0$.  In the remaining case of Example ~\ref{clasped_twisted_fiber.ex} when $i_1\neq 0$ and $i_1\neq i_2,$ $L$ and $L'$  are not Legendrian link-homotopic.

\end{enumerate}
\end{cor}

\begin{rem}
    By choosing certain Legendrian representatives from the isotopy class in Example~\ref{clasped_twisted_fiber.ex}, we see  that all of the above non-simplicity phenomena can occur in a single framed isotopy class. 
\end{rem}

\begin{proof}
Let $V_k$ be a co-orienting vector fied of $\xi_k$.  Let $K_1\sqcup K_2$ be a $V_k$-transverse link in $M$.  Let $L=L_1\sqcup L_2$ be a Legendrian representative of the smooth isotopy class of $K_1\sqcup K_2$, obtained by taking a $C^0$-approximation, is in \cite{etnyre2005legendrian}.  Consider integers $i_1, i_2\in \mathbb{Z}$.  Recall from Section~\ref{background.sec} that, as $V_k$-transverse knots, $(L_1)_{i_1,0}$ is isotopic to $(L_1)_{0,i_1}^{i_1}$. In other words, up to $V_k$-transverse isotopy, $(L_1)_{i_1,0}$ can be obtained from $(L_1)_{0,i_1}$ by adding $i_1$ pairs of kinks as in Figure ~\ref{kink.fig}. Similarly, up to $V_k$-transverse isotopy,
$(L_2)_{i_2,0}$ can be obtained from and $(L_2)_{0,i_2}$ by adding $i_2$ pairs of kinks as in Figure ~\ref{kink.fig}. 

We now consider the constructions in Example ~\ref{parallel_fiber.ex},~\ref{clasped_fiber.ex}, and ~\ref{clasped_twisted_fiber.ex}, with the {\it Legendrian} links $L=(L_1)_{0,i_1}\sqcup(L_2)_{0,i_2}$ and $L'=(L_1)_{i_1,0}\sqcup(L_2)_{i_2,0}$ playing the roles of $K_1\sqcup K_2$ and $K_1^{i_1}\sqcup K_2^{i_2}$. In each case below, we will specify the smooth isotopy class $K_1\sqcup K_2$, and conditions on $i_1$ and $i_2$, and then always define $L=(L_1)_{0,i_1}\sqcup(L_2)_{0,i_2}$ and $L'=(L_1)_{i_1,0}\sqcup(L_2)_{i_2,0}$ as above, based on those choices.

In each case, $L$ and $L'$ are isotopic as framed links and, since every Legendrian immersion in $(M,\xi_k)$ is $V_k$-transverse, $L$ and $L'$ are homotopic as $V_k$-transverse multicurves.  By Theorem~\ref{Vtrans_to_Leg_promotion.thm} Part (1), $L$ and $L'$ are homotopic as Legendrian multi-curves.

Now we turn to the proof of Part (1) of the Corollary.  Suppose that $K_1\sqcup K_2$ is in the isotopy class given in Example ~\ref{parallel_fiber.ex} or Example~\ref{clasped_twisted_fiber.ex}, and assume $i_1\neq i_2.$
Since  $L$ and $L'$ are not link-homotopic ask $V_k$-transverse links, then by Theorem~\ref{Vtrans_to_Leg_promotion.thm} Part (2), they are not link-homotopic as Legendrian links.  Next suppose that $K_1\sqcup K_2$ is in the isotopy class given in Example ~\ref{parallel_fiber.ex}  and $i_1\neq i_2$, or Example~\ref{clasped_twisted_fiber.ex}, with $i_1=0$ and $i_2\neq 0$.  In this setting $L$ and $L'$ are component-wise $V_k$-transverse isotopic. By performing a full Lutz twist along a transverse unknot in the complement of $L$ and $L'$, to get a new contact structure $\xi_k'$ on $M$ which is homotopic to $\xi_k$ as a 2-plane field, we can assume that $L$ and $L'$ have a common overtwisted disk in their complement.  Now by Theorem ~\ref{Vtrans_to_Leg_promotion.thm} Part (3), we have that $L$ and $L'$ are component-wise isotopic as Legendrian links in $(M,\xi_k').$

For the proof of Part (2), suppose $K_1\sqcup K_2$ is in the isotopy class given in Example ~\ref{clasped_fiber.ex} or ~\ref{clasped_twisted_fiber.ex} and assume $i_1\neq 0.$  Since $L$ and $L'$ are not component-wise isotopic as $V_k$-transverse links, they are not component-wise isotopic as Legendrian links in $(M,\xi_k).$  Furthermore, if we choose $K_1\sqcup K_2$ to be the isotopy class in Example~\ref{clasped_fiber.ex} with $i_1\neq 0$ or  Example ~\ref{clasped_twisted_fiber.ex} with $i_1=i_2\neq 0$, $L$ and $L'$ are $V_k$-transverse link-homotopic.  By Theorem~\ref{Vtrans_to_Leg_promotion.thm} Part (2), $L$ and $L'$ are Legendrian link-homtopic.  In the remaining case of Example ~\ref{clasped_twisted_fiber.ex} with $i_1\neq 0$ and $i_1\neq i_2$, $L$ and $L'$ are not $V_k$-transverse link-homotopic, so they are not Legendrian link-homotopic. 
\end{proof}

In Table~\ref{summary.tab} we summarize the results of Section~\ref{examples.sec}.

\begin{table}[htbp]
\resizebox{\textwidth}{!}{
\begin{tabular}{|c||c|c|c|}
\hline
   Links which are  &  Link-homotopic as  &Component-wise isotopic  & Isotopic as \\ 
    $V$-transverse homotopic &$V$-transverse links& as $V$-transverse links& $V$-transverse Links\\
    (resp. Legendrian homotopic) & (resp. as Legendrian links) &(resp. as Legendrian links,&(resp. as Legendrian links, \\
    and framed isotopic&&if components chosen to be loose)&if link chosen to be loose)\\
    \hline\hline
  
     Figure~\ref{fig:parallel_fiber}&&&\\
     $K_1\cup K_2$ and $K_1^{i_1}\cup K_2^{i_2}$, &no&yes&no\\ 
     $i_1\neq i_2$ &&& \\\hline
    Figure~\ref{fig:parallel_fiber}&&&\\
    $K_1\cup K_2$ and $K_1^{i}\cup K_2^{i}$,&yes&yes &yes\\
    $i\neq 0$&&&\\ \hline\hline
     Figure~\ref{fig:clasped_fiber}&&&\\
     $K_1\cup K_2$ and $K_1^{i_1}\cup K_2^{i_2}$, &yes&no&no\\
     $i_1\neq 0$ &&&\\\hline
     Figure~\ref{fig:clasped_fiber}&&&\\
     $K_1\cup K_2$ and $K_1\cup K_2^{i}$, &?&yes&?\\ 
     $i\neq 0$&&&\\\hline\hline
      
      Figure~\ref{fig:clasped_braided}&&&\\
      $K_1\cup K_2$ and $K_1^{i_1}\cup K_2^{i_2}$, &no&no&no\\ 
      $i_1\neq 0$, $i_1\neq i_2$&&&\\ \hline
      Figure~\ref{fig:clasped_braided}&&&\\
      $K_1\cup K_2$ and $K_1\cup K_2^{i}$, &no&yes&no\\
      $i\neq 0$&&& \\\hline
      Figure~\ref{fig:clasped_braided}&&&\\
      $K_1\cup K_2$ and $K_1^i\cup K_2^{i}$, &yes&no&no\\
      $i \neq 0$&&&\\\hline

\end{tabular}}\caption{Summary of non-simplicity phenomena in $S^1$-bundles over orientable surfaces of genus $g\geq 2$.}\label{summary.tab}
\end{table}

\newpage
{\it Acknowledgments: We thank Nikolai Mishacev and Rustam Sadykov for valuable discussions. Patricia Cahn is partially supported by NSF grant DMS-2145384 and Rima Chatterjee acknowledges partial support from SFB/TRR 191 ``Symplectic Structures in Geometry, Algebra and Dynamics, funded by the DFG (Project- ID 281071066-TRR 191).}
\bibliographystyle{alpha}
\bibliography{pseudoLeg}
\end{document}